\documentclass{siamart1116}

\usepackage{lipsum}
\usepackage{amsfonts}
\usepackage{graphicx}
\usepackage{epstopdf}
\usepackage{algorithmic}
\usepackage{mathrsfs}
\usepackage{bm}

\ifpdf
  \DeclareGraphicsExtensions{.eps,.pdf,.png,.jpg}
\else
  \DeclareGraphicsExtensions{.eps}
\fi

\numberwithin{theorem}{section}

\newcommand{\TheTitle}{Multiscale differential {R}iccati equations for linear quadratic regulator problems}
\newcommand{\TheShortTitle}{Multiscale {R}iccati equations for {LQR} problems}
\newcommand{\TheAuthors}{Axel M\r{a}lqvist, Anna Persson and Tony Stillfjord}

\headers{\TheShortTitle}{\TheAuthors}

\title{{\TheTitle}\thanks{Submitted to the editors 2017-06-13.
\funding{This work was supported by the Swedish Research Council under grant no.~2015-04964.}}}

\author{
  Axel M\r{a}lqvist\thanks{Mathematical Sciences, Chalmers University of Technology and the University of Gothenburg
    (\email{axel@chalmers.se}, \email{peanna@chalmers.se}, \email{tony.stillfjord@gu.se}.)}
  \and
  Anna Persson\footnotemark[2]
  \and
   Tony Stillfjord\footnotemark[2]
}

\usepackage{amsopn}

\ifpdf
\hypersetup{
  pdftitle={\TheTitle},
  pdfauthor={\TheAuthors}
}
\fi

\newcommand{\xb}{\bm{x}}
\newcommand{\yb}{\bm{y}}
\newcommand{\Ab}{\bm{A}}
\newcommand{\Bb}{\bm{B}}
\newcommand{\Cb}{\bm{C}}
\newcommand{\Db}{\bm{D}}

\newcommand{\Gb}{\bm{G}}
\newcommand{\Ib}{\bm{I}}
\newcommand{\Lb}{\bm{L}}
\newcommand{\Mb}{\bm{M}}

\newcommand{\Qb}{\bm{Q}}
\newcommand{\Rb}{\bm{R}}
\newcommand{\Sb}{\bm{S}}
\newcommand{\Ub}{\bm{U}}
\newcommand{\Vb}{\bm{V}}
\newcommand{\Wb}{\bm{W}}
\newcommand{\Xb}{\bm{X}}

\newcommand{\Ac}{\mathcal{A}}
\newcommand{\Bc}{\mathcal{B}}
\newcommand{\Cc}{\mathcal{C}}

\newcommand{\Gc}{\mathcal{G}}
\newcommand{\Lc}{\mathcal{L}}
\newcommand{\Qc}{\mathcal{Q}}
\newcommand{\Rc}{\mathcal{R}}
\newcommand{\Tc}{\mathcal{T}}

\newcommand{\Fs}{\mathscr{F}}
\newcommand{\Gs}{\mathscr{G}}

\renewcommand{\phi}{\varphi}
\newcommand{\LOD}{\text{ms}}
\newcommand{\f}{\text{f}}
\newcommand{\ALOD}{\Ac^\LOD}
\newcommand{\BLOD}{\Bc^\LOD}
\newcommand{\CLOD}{\Cc^\LOD}
\newcommand{\XLOD}{X^\LOD}
\newcommand{\Xt}{\tilde{X}}
\newcommand{\Yt}{\tilde{Y}}

\newcommand{\LtwoOmega}{L^2(\Omega)}
\newcommand{\Ltwo}{L^2}
\newcommand{\VLOD}{V_\LOD}
\newcommand{\LHH}{\Lc(\Ltwo)}
\newcommand{\LVV}{\Lc(V)}

\newcommand{\iprod}[2]{\left( #1, #2 \right)}
\newcommand{\norm}[1]{\lVert #1 \rVert}
\newcommand{\bignorm}[1]{\big\lVert #1 \big\rVert}

\newcommand{\LHnorm}[1]{\norm{#1}_{\LHH}}
\newcommand{\bigLHnorm}[1]{\bignorm{#1}_{\LHH}}
\newcommand{\LVnorm}[1]{\norm{#1}_{\LVV}}

\DeclareMathOperator{\Id}{Id}
\newcommand{\IdHh}{\Id_H^h}
\newcommand{\IbHh}{\Ib_H^h}
\newcommand{\IdLODh}{\Id_{\LOD}^h}

\newcommand{\diff}[1]{\mathrm{d}#1}
\newcommand{\ds}{\diff{s}}
\newcommand{\dt}{\diff{t}}

\newcommand{\R}{\mathbb{R}}
\renewcommand{\exp}[1]{\mathrm{e}^{#1}}
\newcommand{\domain}[1]{\mathcal{D}(#1)}
\DeclareMathOperator*{\essinf}{ess\,inf}
\DeclareMathOperator*{\esssup}{ess\,sup}
\DeclareMathOperator{\diam}{diam}
\DeclareMathOperator{\interior}{int}

\newsiamthm{assumption}{Assumption}
\newsiamremark{remark}{Remark}

\begin{document}

\maketitle

\begin{abstract}
We consider approximations to the solutions of differential Riccati equations in the context of linear quadratic regulator problems, where the state equation is governed by a multiscale operator. Similarly to elliptic and parabolic problems, standard finite element discretizations perform poorly in this setting unless the grid resolves the fine-scale features of the problem. This results in unfeasible amounts of computation and high memory requirements. In this paper, we demonstrate how the localized orthogonal decomposition method may be used to acquire accurate results also for coarse discretizations, at the low cost of solving a series of small, localized elliptic problems. We prove second-order convergence (except for a logarithmic factor) in the $L^2$ operator norm, and first-order convergence in the corresponding energy norm. These results are both independent of the multiscale variations in the state equation. In addition, we provide a detailed derivation of the fully discrete matrix-valued equations, and show how they can be handled in a low-rank setting for large-scale computations. In connection to this, we also show how to efficiently compute the relevant operator-norm errors. Finally, our theoretical results are validated by several numerical experiments.
\end{abstract}






\section{Introduction}
In a linear quadratic regulator (LQR) problem, the state $x$ is a model of a system whose evolution can be influenced through the input $u$. The goal is to drive certain measurable quantities of the system, the output $y$, to a given target which is typically zero. The relations between $x$, $u$ and $y$ are given by the state and output equations
\begin{align}
  \dot{x} &= \Ac x + \Bc u, \quad x(0) = x_0, \label{eq:state}\\
  y &= \Cc x, \label{eq:output}
\end{align}
where $\Ac$, $\Bc$ and $\Cc$ are given operators. The optimal input function $u^*$ is found by minimizing the cost functional
\begin{equation*}
  J(u) = \int_{0}^{T} { \iprod{\Qc y}{y} + \iprod{\Rc u}{u} \,\dt} + \iprod{\Gc y(T)}{y(T)},
\end{equation*}
where $\Qc$, $\Rc$ and $\Gc$ are given weighting factors. It can be shown (see e.g.~\cite{AbouKandil_etal2003, LasieckaTriggiani2000}) that $u^*$ is given in feedback form as $u^*(t) = - \Rc^{-1} \Bc^* X(T-t) x(t)$, where $X$ is the solution to an operator-valued differential Riccati equation (DRE):
\begin{align} \label{eq:DRE}
  \begin{aligned}
    \dot{X}(t) &= \Ac^{*}X(t) + X(t)\Ac + \Cc^*\Qc\Cc - X(t)\Bc\Rc^{-1}\Bc^*X(t),\\
          X(0) &= \Gc.
  \end{aligned}
\end{align}
In the case of a nonzero output target, one additional differential equation for the evolution of $u^*$ has to be solved.

In this paper, we consider the case when the operator $\Ac$ exhibits multiscale behaviour. In particular, we consider diffusion problems where the spatial variation of the diffusion coefficient is on a fine scale compared to the computational domain. This e.g.\ occurs in the modeling of composite materials and flows in porous media. Numerically approximating the solutions to elliptic or parabolic equations given by such operators in the usual way is difficult, because a very fine discretization is necessary to resolve the fine-scale structure. These difficulties are exacerbated when considering DREs such as \cref{eq:DRE}, as their solution essentially requires solving many parabolic equations.

A by now well established method for multiscale elliptic and parabolic problems is the localized orthogonal decomposition (LOD)~\cite{MalqvistPeterseim_2014,HenningMalqvist2014}. It is a modification of the finite element method (FEM), which incorporates some of the fine-scale structure into a coarse discretization by precomputing a series of localized fine-scale problems. Due to the localization, these are much cheaper to evaluate than the full fine-scale problem and may additionally be solved in parallel.

We note that finite elements were introduced for the approximation of optimal control problems already in the 1970's, see e.g.~\cite{McKnightBosarge1973,BosargeJohnsonSmith1973,Falk1973,Winther1978}, and the field has grown much in several different directions since then. When diffusion problems have been considered, the focus has typically been on constant or slowly varying diffusion. Recently, however, also optimal control problems of multiscale type have been considered in e.g.~\cite{CaoLiuAllegrettoLin2012,ChenHuangLiuYan2015,Li2010}.
 None of these consider the LOD approach, instead preferring homogenization or asymptotic expansions. Additionally, a common assumption is that the multiscale features are periodic, which is frequently not the case in applications.

The focus in this paper is on the approximation of DREs such as \cref{eq:DRE}. In contrast to the forward-adjoint approach, which solves a specific optimal control problem, the DRE provides the feedback laws for all problems defined by the operators $\Ac, \Bc, \Cc$. While more expensive to solve, it can be precomputed and reused in many different situations. We refer to~\cite{Bensoussan_etal_2007,LasieckaTriggiani2000} for an overview of Riccati theory, with the latter reference treating very general problems.

Our main result is that LOD-approximations to the solution of \cref{eq:DRE} with a mesh size $H$ converge with order $H^2\log(H^{-1})$ in
the $\Ltwo$ operator norm to a given accurate fine-scale FEM approximation.
The convergence in the corresponding operator energy norm is shown to be of order $H$.
We note that $H^2\log(H^{-1})$-convergence of FEM approximations to the exact solution of \cref{eq:DRE} has previously been shown in~\cite{KrollerKunisch1991}, and similar results for algebraic Riccati equations can be found in~\cite{LasieckaTriggiani2000}. (See also~\cite{Rosen1991,BennerMena2016} for convergence results without orders in related settings.)  However, the error constants in these results depend on the multiscale variations of $\Ac$, and thus such convergence is not observed in practice. This is not the case for our present results.

For practical computations, also a temporal discretization is necessary; for this we consider a low-rank splitting scheme as introduced in~\cite{Stillfjord2015}. Such methods decompose the DRE into its affine and nonlinear parts and approximate these separately, thereby greatly reducing the computational cost. The affine problem requires the approximation of several parabolic equations involving $\Ac$ in each time step.
 As the computational efficiency gain for LOD increases with the number of times the modified basis may be reused, splitting schemes are thus particularly well suited to be combined with the LOD method.

We demonstrate how to transform the FEM and LOD discretizations into matrix-valued equations, and how to implement the fully discrete methods. Even if LOD reduces the need for very fine discretizations, large 2D or 3D-problems may still yield large matrices. We therefore consider the low-rank approach, which greatly reduces the necessary amount of computations. As a side effect, this also allows us to compute errors in the operator norms very efficiently.

A brief outline of the paper is as follows: We formalize the setting and our basic assumptions in \cref{sec:setting}, and define the different spatial discretizations in \cref{sec:space_discr}. Convergence of the LOD approximations with the appropriate order is then shown in \cref{sec:errors}. The matrix-valued formulations of the discretized DREs and related questions are discussed in \cref{sec:matrices}, while \cref{sec:time} is devoted to the temporal discretization and low-rank setting. Finally, we present several numerical experiments and their results in \cref{sec:experiments}.

\section{Setting} \label{sec:setting}
Let $\Omega \in \R^d$, $d \le 3$, be a bounded polygonal/polyhedral domain. We consider the separable Hilbert spaces $\LtwoOmega$, $V = H^1_0(\Omega)$, $U$ and $Z$, where $\LtwoOmega$ corresponds to the state space, $U$ is the control space and $Z$ is the observation space. In the following, the specification of $\Omega$ will be omitted. We write $\iprod{\cdot}{\cdot}$ and $\norm{\cdot}$ for the inner product and norm on $\Ltwo$, and denote the corresponding quantities on $V$, $U$ and $Z$ by subscripts. To define the state evolution operator $\Ac$, we assume that the inner product $a(u,v) = \int \kappa \nabla u \cdot \nabla v$ on $V \times V$ is given, with assumptions on $\kappa$ given below. Then $\Ac \colon \Ltwo \supset \domain{\Ac} \to \Ltwo$ is defined by $\iprod{\Ac u}{v} = -a(u,v)$ and $\domain{\Ac} = \{u \in V \;|\; \Ac u \in \Ltwo \}$.

 Further, let the input operator $\Bc \colon U \to \Ltwo$ and the output operator $\Cc \colon \Ltwo \to Z$ be given. We also consider the output and input weighting operators $\Qc \colon Z \to Z$ and $\Rc \colon U \to U$ (which could be included in $\Cc$ and $\Bc$ but are typically not) and the final state weighting operator $\Gc \colon \Ltwo \to \Ltwo$. By $^*$, we denote Hilbert-adjoint operators with respect to $\Ltwo$, so that e.g.\ $\Bc^*\colon \Ltwo \to U$ satisfies $\iprod{\Bc x}{y} = \iprod{x}{\Bc^*y}$ for all $x \in U$ and $y \in \Ltwo$. Finally, we denote the linear bounded operators from one generic Hilbert space, $Y$, to another, $W$, by $\Lc(Y,W)$. When $W=Y$, we abbreviate $\Lc(Y) = \Lc(Y,Y)$.

In this notation, the weak form of \cref{eq:DRE} is to find $X \in \LHH$ satisfying
\begin{equation}
  \label{eq:DRE_weak}
  \iprod{\dot{X}x}{y} = \iprod{Xx}{\Ac y} + \iprod{Xy}{\Ac x} + \iprod{\Qc\Cc x}{\Cc y}_Z - \iprod{\Rc^{-1} \Bc^* X x}{\Bc^* X y}_U ,
\end{equation}
for all $x,y \in \domain{\Ac}$.

\begin{assumption} \label{ass:operators}
The diffusion coefficient $\kappa \in L^\infty(\R^{d \times d})$ is symmetric and satisfies
\begin{align*}
  0 &< \alpha := \essinf_{x \in \Omega} \inf_{v\in \R^d \setminus \{0\}} \frac{\kappa(x) v \cdot v}{v \cdot v},\\
\infty &> \beta := \esssup_{x \in \Omega} \sup_{v\in \R^d \setminus \{0\}} \frac{\kappa(x) v \cdot v}{v \cdot v}.
\end{align*}
In addition, $\Bc \in \Lc(U, \Ltwo)$, $\Cc \in \Lc(\Ltwo, Z)$, $\Qc \in \Lc(Z)$, $\Rc \in \Lc(U)$ is invertible with $\Rc^{-1} \in \Lc(U)$ and $X(0) =  \Gc \in \Lc(\Ltwo)$.
\end{assumption}
The first part of \cref{ass:operators} shows that $a$ is a bounded and coercive bilinear form, which means that $\Ac$ is the generator of an analytic semigroup $\exp{t\Ac} \colon \Ltwo \to \Ltwo$, see e.g.~\cite[Theorem 3.6.1]{Tanabe1979}. In conjunction with the boundedness assumptions on $\Bc$, $\Cc$, $\Qc$ and $\Rc$, this guarantees the existence and uniqueness of a solution to \cref{eq:DRE_weak}. In fact, there is even a classical solution to \cref{eq:DRE}~\cite[Part IV, Ch.~3]{Bensoussan_etal_2007}, which means that the $\Ac^*X + X\Ac$ term can be extended to an operator in $\LHH$. As a consequence, Equation \cref{eq:DRE_weak} holds also for $x, y \in \Ltwo$. We note that these conclusions are valid also under various weaker forms of \cref{ass:operators}, which additionally permit the treatment of boundary control and observation~\cite{LasieckaTriggiani2000}. A discussion on an extension of our results to such a setting may be found in \cref{subsec:BCC}.

\section{Spatial discretization}\label{sec:space_discr}
We first introduce the FEM approximation of \cref{eq:DRE_weak}. To this end, we let $\Tc_h$ be a triangulation of $\Omega$ with meshwidth $h$ and $N_h$ internal nodes. The subspace $V_h \subset V$ denotes the space of continuous and piecewise affine functions on $\Tc_h$, and we denote the corresponding nodal basis functions by $\{\phi^h_i\}_{i=1}^{N_h}$. This discretization is referred to as the fine, or sometimes also reference, mesh, see further \cref{sec:space_discr_LOD} below.

We also consider a coarse discretization space $V_H \subset V_h$ for $H>h$, with the corresponding family of triangulations $\{\Tc_H\}_{H>h}$, which is assumed to be quasi-uniform. For these triangulations, we let $B_K$ be the largest ball contained in the triangle $K$ and denote by $\gamma>0$ the shape regularity of the mesh, defined by
\begin{align*}
	\gamma := \max \gamma_K, \quad \gamma_K := \frac{\diam B_K}{\diam K}, \quad \forall K \in \Tc_{H}, \ H>h.
\end{align*}
Furthermore, we let $\IdHh\colon V_H \to V_h$ denote the identity operator between these spaces, i.e.\ $\IdHh u = u$ for all $u \in V_H$. Similarly, $\Id_h \colon V_h \to \Ltwo$ is the identity operator mapping into $\Ltwo$ and $\Id_h^*$ is the $\Ltwo$-orthogonal projection of $\Ltwo$ onto $V_h$.

The semi-discretized weak form of~\cref{eq:DRE} is defined by
\begin{equation} \label{eq:DRE_FEM}
  \iprod{\dot{X}_h x}{y} = \iprod{X_h x}{\Ac_h  y} + \iprod{X_h y}{\Ac_h  x} + \iprod{\Qc\Cc_h x}{\Cc_h y}_Z - \iprod{\Rc^{-1} \Bc_h^* X_h x}{\Bc_h^* X_h y}_U .
\end{equation}
for all $x, y \in V_h$ and with $X_h \colon V_h \to V_h$ satisfying $X_h(0) = \Id_h^* X(0) \Id_h$. Here, the operators $\Ac_h  \colon V_h \to V_h$, $\Bc_h \colon U \to V_h$ and $\Cc_h \colon V_h \to Z$ satisfy
\begin{equation*}
  \iprod{\Ac_h  x}{y} =  \iprod{\Ac x}{y}, \quad  \iprod{\Bc_h u}{y} =  \iprod{\Bc u}{y} \quad \text{and} \quad \iprod{\Cc_h x}{z} =  \iprod{\Cc x}{z}
\end{equation*}
for all $x,y \in V_h$, $u \in U$ and $z \in Z$.
We note that $X$ can be proven to be self-adjoint, so we additionally require that $X_h$ is self-adjoint.

For the coarse discretization, we have the same equation but with $H$ instead of $h$. We observe that the coarse and fine operators are related in the following way:
\begin{equation}\label{eq:fine_coarse_operator_relations}
  \Ac_H = (\IdHh)^* \Ac_h  \IdHh, \quad \Bc_H = (\IdHh)^* \Bc_h \quad \text{and} \quad \Cc_H = \Cc_h \IdHh,
\end{equation}
and that the natural extension of $X_H$ to a map on $V_h$ is given by $\IdHh X_H (\IdHh)^*$. Here, $(\IdHh)^*$ is the $\Ltwo$-orthogonal projection of $V_h$ onto $V_H$.

\subsection{Localized orthogonal decomposition}\label{sec:space_discr_LOD}
If $\kappa$ is varying on a small scale of size $\epsilon>0$, then the classical FEM approximation of a parabolic problem $\dot{x} = \Ac x + f$ may yield poor results, unless $h$ is sufficiently small to resolve the fine-scale variations. That is, we typically do not observe $O(h^2)$-convergence until $h < \epsilon$, which requires infeasible amounts of computation. The same behaviour occurs for the $X_h$-discretizations of~\cref{eq:DRE_weak}.

To this end, we assume that $h$ is sufficiently small so that $X_h$ is a good approximation of $X$. That is, $h < \epsilon$, and we refer to $X_h$ as the reference solution. The aim is now to approximate $X_h$ by using a multiscale space $V_\LOD\subset V_h$ of the same dimension as the coarse space $V_H$. To obtain such a space, we use the localized orthogonal decomposition (LOD) method introduced in \cite{MalqvistPeterseim_2014}, which incorporates fine-scale information in the coarse-scale space. The construction involves the solution of several fine-scale, but localized and parallelizable, problems. We briefly summarize the procedure here and refer to \cite{MalqvistPeterseim_2014, HenningMalqvist2014}, for the details.

To define the multiscale space $V_\LOD$, we first introduce an interpolation operator $I_{H}\colon V_h \rightarrow V_{H}$ that fulfills
\begin{align*}
H^{-1}\norm{v-I_H v}_{L^2(K)} + \norm{\nabla I_H v}_{L^2(K)} \le C \norm{\nabla v}_{L^2(\omega_K)}, \forall v \in V_h,
\end{align*}
for all triangles $K\in \Tc_H$, where $\omega_K :=\cup \{\hat{K} \in \Tc_H: \hat{K} \cap K \neq \emptyset \}$. In this paper we use the weighted Cl\'{e}ment interpolant as in \cite{MalqvistPeterseim_2014}. Let $V_\f$ denote the kernel of $I_H$,
\begin{align*}
V_\f := \ker I_H = \{v \in V_h: I_Hv=0 \},
\end{align*}
and note that $V_h$ can be decomposed as $V_h = V_H \oplus V_\f$, meaning that every $v_h \in V_h$ can be written as $v_h = v_H + v_\f$ with $v_H\in V_H, v_\f \in V_\f$. We now introduce the (global) correction operator $\hat{Q}_h \colon V_H \rightarrow V_\f$ by
\begin{align*}
a(\hat{Q}_hv,w) = a(v,w), \quad \forall w \in V_\f,
\end{align*}
and define the (global) multiscale space as $\hat{V}_\LOD := \hat{R}_hV_H = V_H - \hat{Q}_h V_H$, with $\hat{R}_h:=\IdHh - \hat{Q}_h$. This leads to the decomposition $V_h = \hat{V}_\LOD \oplus V_\f$ with the orthogonality $a(\hat{v}_\LOD, v_\f)=0$, $\hat{v}_\LOD \in \hat{V}_\LOD$, $v_\f \in V_\f$. Note that $\hat{Q}_h$ is the orthogonal projection onto $V_\f$ with respect to the inner product $a(\cdot,\cdot)$, i.e. the Ritz projection onto $V_\f$, and $\hat{V}_\LOD$ is the orthogonal complement to $V_\f$. From the construction it follows that $\dim \hat{V}_\LOD = \dim V_H$. Indeed, a basis for $\hat{V}_\LOD$ is given by $\{\varphi^H_i - \hat{Q}_h \varphi^H_i: i=1,..., N_H \}$.

In general, the corrections $\hat{Q}_h \varphi^H_i$  have global support and are expensive to compute, since they are posed in the entire fine scale space $V_\f \subseteq V_h$. To overcome this, it is observed that the corrections have exponential decay away from the $i$:th node of $\Tc_H$ (see \cite{MalqvistPeterseim_2014, HenningMalqvist2014}), which motivates a truncation of the corrections. For this purpose, we define patches $\omega_k(K)$ of size $k$ around each $K \in \Tc_H$ by the following:
\begin{align*}
\omega_0(K) &:= \interior{K}, \\
\omega_k(K) &:= \interior{ \big(\cup \{\hat{K} \in \Tc_H : \hat{K} \cap \overline{\omega_{k-1}(K)} \neq \emptyset\}\big)}, \quad k=1,2,....
\end{align*}
Further, we define $V^K_{\f }:=\{v\in V_\f :v(z)=0 \text{ on }\overline{\Omega}\setminus \omega_k(K)\}$ to be the restriction of $V_\f$ to the patch $\omega_k(K)$. For brevity, we do not include the dependence on $k$ in the notation. Now note that the correction operator $\hat{Q}_h$ can be written as the sum $\hat{Q}_h = \sum_{K\in \Tc_H} \hat{Q}_h^K,$ where
\begin{align*}
a(\hat{Q}_h^K v,w)=\int_K \kappa \nabla v \cdot \nabla w , \quad  \forall w \in V_\f, \ v \in V_H,\ K \in \Tc_H.
\end{align*}
We can now localize these computations by replacing $V_\f$ with $V_{\f }^K$. Define $Q_h^K \colon V_H \rightarrow V_{\f }^K$ such that
\begin{align*}
a(Q_h^Kv,w) = \int_K \kappa \nabla v \cdot \nabla w , \quad \forall w \in V_{\f }^K, \ v \in V_H,\ K \in \Tc_H.
\end{align*}
Finally, we can define a local operator $Q_h := \sum_{K\in \Tc_H} Q_h^K$ and a localized space $V_\LOD := R_h V_H = V_H - Q_h V_H$, with $R_h := \IdHh - Q_h$.

The approximation properties (and the required computational effort) of the space $V_\LOD$ depends on the choice of $k$. In \cite{HenningMalqvist2014} it is proven that convergence of order $H^2$ is obtained if $k$ is chosen proportional to $\log H^{-1}$. In this paper we therefore assume that $k \sim \log H^{-1}$ to avoid explicitly stating the dependence on $k$.

To define an LOD-approximation to the solution $X_h$ in~\cref{eq:DRE_FEM}, we additionally need to introduce the identity operator $\IdLODh : \VLOD \to V_h$, $\IdLODh u = u$. Its $\Ltwo$-adjoint is the $\Ltwo$-orthogonal projection of $V_h$ onto $\VLOD$. Replacing the space $V_h$ with $V_\LOD$ then results in the problem to find $\XLOD_h \colon \VLOD \to \VLOD$ satisfying
\begin{equation} \label{eq:DRE_LOD}
  \begin{aligned}
  \iprod{\dot{X}^{\LOD}_h u}{v} &= \iprod{\XLOD_h u}{\ALOD_h  v} + \iprod{\XLOD_h v}{\ALOD_h  u} \\
&+ \iprod{\Qc\CLOD_h u}{\CLOD_h v}_Z - \iprod{\Rc^{-1} (\BLOD_h)^* \XLOD_h u}{(\BLOD_h)^* \XLOD_h v}_U
 \end{aligned}
\end{equation}
for all $u, v \in \VLOD$ and the initial condition $\XLOD_h(0) = (\IdLODh)^*\Id_h^* X(0) \Id_h \IdLODh$.
Here, the operators $\ALOD_h \colon \VLOD \to \VLOD$, $\BLOD_h \colon U \to \VLOD$ and $\CLOD_h \colon \VLOD \to Z$ are given by
\begin{align*}
 \iprod{\ALOD_h v}{w} = \iprod{\Ac v}{w}, \; \iprod{\BLOD_h u}{w}_U = \iprod{\Bc u}{w}_U \; \text{and} \; \iprod{\CLOD_h v}{z}_Z = \iprod{\Cc u}{z}_Z
\end{align*}
for all $v, w \in \VLOD$, $u \in U$ and $z \in Z$. Similar to \cref{eq:fine_coarse_operator_relations} we have
\begin{equation}\label{eq:fine_LOD_operator_relations}
  \ALOD_h = (\IdLODh)^* \Ac_h  \IdLODh, \quad \BLOD_h = (\IdLODh)^* \Bc_h \quad \text{and} \quad \CLOD_h = \Cc_h \IdLODh.
\end{equation}
The natural $V_h$-extension of $\XLOD_h$ is given by $\IdLODh \XLOD_h (\IdLODh)^*$, similar to the $X_H$-case.

Since $\VLOD$ has the same dimension as $V_H$, there is a lower-dimensional representative for $\XLOD_h$, given by $\XLOD_h = R_h \XLOD_H R_h^{-1}$. By inserting $u = R_h x$ and $v = R_h y$, with $x,y \in V_H$, in \cref{eq:DRE_LOD} we see that
  \begin{align*}
    \iprod{\dot{X}^{\LOD}_H x}{R_h^* R_h y} &= \iprod{\XLOD_H x}{R_h^* \Ac_h  R_h y} + \iprod{\XLOD_H y}{R_h^* \Ac_h  R_h x} \\
&+ \iprod{\Qc\Cc_h R_h x}{\Cc_h R_h y}_Z - \iprod{\Rc^{-1} \Bc_h^* R_h \XLOD_H x}{\Bc_h^* R_h \XLOD_H y}_U  ,
  \end{align*}
and we consequently define the corrected coarse-scale operators
\begin{equation*}
  \ALOD_H = R_h^* \Ac_h  R_h, \quad \BLOD_H = R_h^* \Bc_h \quad \text{and} \quad \CLOD_H = \Cc_h R_h.
\end{equation*}

\section{Error analysis} \label{sec:errors}
In the following, $C$ denotes a generic constant which may take different values at different occasions. It may depend on the problem data and the size of the domain, but is independent of $h$ and $H$. Moreover, it does not depend on the multiscale variations of $\Ac$, i.e.\ any derivatives of $\kappa$. We start by gathering some useful results:

\subsection{Preliminaries}
Recall that $\Id_h \colon V_h \to \Ltwo$ is the identity mapping, $P_h = \Id_h^* \colon \Ltwo \to V_h$ denotes the $\Ltwo$-orthogonal projection onto $V_h$, and $P_\LOD$ is the $\Ltwo$-orthogonal projection onto $\VLOD$. We have $P_\LOD = (\IdLODh)^* P_h$, i.e.\ we first project onto $V_h$ and then onto $\VLOD$. Straightforward calculations show the following:
\begin{lemma}\label{lemma:bounded_IO}
Under \cref{ass:operators}, it holds that $\Id_h \Bc_h \in \Lc(U,\Ltwo)$, $\Cc_h P_h \in \Lc(\Ltwo,Z)$ and $S_h := \Id_h \Bc_h \Rc^{-1} \Bc_h^* \Id_h^* \in \LHH$.
  \end{lemma}
Further, let $\exp{t\Ac_h }$ denote the solution operator to the equation $\dot{u} + \Ac_h u = 0$, i.e.\ the semigroup generated by $\Ac_h $. Similarly, $\exp{t\ALOD_h}$ is the semigroup generated by $\ALOD_h$. Because $\Ac$ generates an analytic semigroup on $\Ltwo$, these operators are analytic semigroups on $V_h$ and $V_\LOD$, respectively. More specifically, we have
\begin{lemma} \label{lemma:bounded_solop}
Under \cref{ass:operators}, the operators
\begin{equation*}
  E_h(t) = \Id_h \exp{t\Ac_h } \Id_h^* \quad \text{and} \quad
  E_\LOD(t) = \Id_h \IdLODh \exp{t\ALOD_h} (\IdLODh)^* \Id_h^*,
\end{equation*}
are both in $\LHH$ for $t \in [0,T]$, with the uniform bounds $\LHnorm{E_h(t)} \le 1$ and $\LHnorm{E_\LOD(t)} \le 1$.
\end{lemma}
By arguing as in~\cite{MalqvistPersson_2015}, but for the (simpler) semi-discrete case, we have (choosing $k \sim \log H)$
\begin{lemma} \label{lemma:parabolic_LOD}
For $t \in (0,T]$ it holds that
 \begin{equation*}
   \LHnorm{E_h(t) - E_\LOD(t)} \le C H^2 t^{-1}.
 \end{equation*}
Here, the constant $C$ depends on $T$, $\alpha$, $\beta$, and $\gamma$, but not on the multiscale variations of $\Ac$.
\end{lemma}
\begin{proof}
  We only comment briefly on the proof here.
  Let $u_h(t) = e^{t\Ac_h}P_hv$ and $u_\LOD(t)=e^{t\ALOD_h}P_\LOD v$, for $v\in \LtwoOmega$. By introducing the Ritz projection $R_{\LOD}\colon V_h \rightarrow \VLOD$ satisfying $a(R_\LOD v,w)=a(v,w)$ for all $w \in \VLOD, v \in V_h$ we get, see \cite[Chapter 3]{Thomee2006} and \cite{MalqvistPersson_2015},
  \begin{align*}
    \norm{u_h - u_\LOD} \le Ct^{-1}\sup_{s\le t} \Big\{ s^2 \norm{\dot\rho} + s \norm{\rho} + \norm{ \int_0^s\rho(r) \,\diff{r} } \Big\},
  \end{align*}
  where $\rho:=u_h - R_\LOD u_h$. From the error bounds of $R_\LOD$ in \cite{MalqvistPeterseim_2014}, see also \cite{MalqvistPersson_2015}, we get
  \begin{equation*}
    \norm{u_h - u_\LOD} \le CH^2t^{-1}\sup_{s\le t} \big(s^2 \norm{\ddot u_h(s)} + s \norm{\dot u_h(s)} + \norm{u_h(s)} + \norm{v} \big).
  \end{equation*}
  The regularity estimates $\norm{D_t^l u_h(t)} \le Ct^{-l} \norm{v}$, for $t=0,1,2$, \cite[Lemma 2.5]{Thomee2006}, completes the proof.
\end{proof}
Finally, from \cref{lemma:bounded_IO}, we get the existence and uniqueness of solutions $X_h$ and $\XLOD_h$ to the discretized DREs \cref{eq:DRE_FEM} and \cref{eq:DRE_LOD}, respectively.
Let us abbreviate
\begin{equation*}
  \Xt(t) = \Id_h X_h(t) \Id_h^* \quad \text{and} \quad
  \Yt(t) = \Id_h \IdLODh \XLOD_h (\IdLODh)^* \Id_h^*.
\end{equation*}
Then we have
\begin{lemma}\label{lemma:bounded_sol}
  There is a constant $C > 0$ which is independent of the multiscale variations of $\Ac$ but may depend on $\alpha$ and $\beta$, such that
  \begin{equation*}
    \LHnorm{\Xt(t)} + \LHnorm{\Yt(t)} \le C,
  \end{equation*}
  for $t\in [0,T]$.
\end{lemma}

\subsection{Error analysis}
We are now ready for the main theorem of this paper:
\begin{theorem} \label{thm:main_error_L2}
Suppose that \cref{ass:operators} is fulfilled. Then for $t \in (0,T]$ it holds that
 \begin{equation*}
\LHnorm{\Xt(t) - \Yt(t)} \le C H^2 \big( \log (H^{-1}) +  t^{-1} \big).
\end{equation*}
Here, the constant $C$ depends on $T$, $\alpha$, $\beta$, $\gamma$, and $\LHnorm{X(0)}$, but not on the multiscale variations of $\Ac$.
\end{theorem}
\begin{proof}
We utilize the integral form of  \cref{eq:DRE_FEM}. If $X_h$ solves \cref{eq:DRE_FEM} then it satisfies
\begin{equation} \label{eq:DRE_integral}
  X_h(t) = \exp{t\Ac_h^*} X_h(0) \exp{t\Ac_h } + \int_{0}^{t}{\exp{(t-s)\Ac_h^*}\Big( \Cc_h^*\Qc\Cc_h - X_h(s) \Bc_h \Rc^{-1} \Bc_h^* X_h(s) \Big) \exp{(t-s)\Ac_h }\,\ds}.
\end{equation}
(see e.g.~\cite[Chapter IV:3, Proposition 2.1]{Bensoussan_etal_2007}).
Recalling that  $\Id_h^*\Id_h$ and $(\IdLODh)^*\IdLODh$ are the identity operators on $V_h$ and $V_\LOD$, respectively, and using \cref{eq:fine_LOD_operator_relations} therefore shows that
\begin{align*}
  \Xt(t) &= E_h(t)^* \Xt(0) E_h(t) \\
         &\quad+ \int_{0}^{t}{E_h(t-s)^* \Big( (\Cc_hP_h)^*\Qc\Cc_hP_h - \Xt(s) S_h \Xt(s) \Big) E_h(t-s) \,\ds},
\end{align*}
as well as
\begin{align*}
  \Yt(t) &= E_\LOD(t)^* \Xt(0) E_\LOD(t) \\
         &\quad+ \int_{0}^{t}{E_\LOD(t-s)^* \Big( (\Cc_hP_h)^*\Qc\Cc_hP_h - \Yt(s) S_h \Yt(s) \Big) E_\LOD(t-s) \,\ds}.
\end{align*}

(Note the $\Xt(0)$ in the first term, since we suppose $X^\LOD_h(0) = (\IdLODh)^* X_h(0) \IdLODh$.)

Subtracting these expressions yields
\begin{align*}
  \Xt(t) - \Yt(t) &= E_h(t)^*\Xt(0) \Big( E_h(t) - E_\LOD(t) \Big) + \Big( E_h(t) - E_\LOD(t)  \Big)^* \Xt(0) E_\LOD(t) \\
&\quad + \int_0^t E_h(t-s)^* (\Cc_hP_h)^*\Qc(\Cc_hP_h) \Big( E_h(t-s) - E_\LOD(t-s) \Big) \\
&\qquad + \Big( E_h(t-s) - E_\LOD(t-s)  \Big)^*(\Cc_hP_h)^*\Qc(\Cc_hP_h) E_\LOD(t-s) \\
&\qquad + \Big( E_h(t-s) - E_\LOD(t-s)  \Big)^* \Xt(s) S_h \Xt(s) E_h(t-s) \\
&\qquad + E_\LOD(t-s)^* \Xt(s) S_h \Xt(s) \Big( E_h(t-s) - E_\LOD(t-s)  \Big) \\
&\qquad + E_\LOD(t-s)^* \Big( \Xt(s) - \Yt(s) \Big) S_h \Xt(s) E_\LOD(t-s) \\
&\qquad + E_\LOD(t-s)^* \Yt(s) S_h \Big( \Xt(s) - \Yt(s) \Big) E_\LOD(t-s) \,\ds \\
&=: R_1 + R_2 + \int_0^t \sum_{j=3}^8{R_j(s)} \, \ds,
\end{align*}
so that
\begin{equation*}
  \LHnorm{\Xt(t) - \Yt(t)} \le \LHnorm{R_1} + \LHnorm{R_2} + \int_0^t \bigLHnorm{ \sum_{j=3}^8{R_j(s)}} \, \ds.
\end{equation*}

We observe that for all $G \colon \Ltwo \to Y$ (with a generic Hilbert space $Y$) it holds that $\norm{G}_{\Lc(\Ltwo, Y)} = \norm{G^*}_{\Lc(Y, \Ltwo)}$. Thus, using \cref{lemma:bounded_solop,lemma:parabolic_LOD,lemma:bounded_sol} we get
\begin{equation*}
  \LHnorm{R_1} = \LHnorm{R_2} \le C H^2 t^{-1} \LHnorm{\Xt(0)} \le C H^2 t^{-1} ,
\end{equation*}
Additionally using \cref{lemma:bounded_IO} shows that the last two integrands satisfy
\begin{align*}
  \LHnorm{R_7(s) + R_8(s)} &\le C \Big(\LHnorm{\Xt(s)} + \LHnorm{\Yt(s)} \Big)  \LHnorm{\Xt(s) - \Yt(s)} \\
                                    &\le C \LHnorm{\Xt(s) - \Yt(s)}.
\end{align*}
Due to the singularity at $s=t$ in the bound on $\LHnorm{E_h(t-s) - E_\LOD(t-s)}$, we split the integrals of the remaining $R_j$-terms into two parts. For $R_3$, we find
\begin{align*}
\int_{0}^{t} \LHnorm{R_3(s)} \, \ds &\le \int_{0}^{t - H^2} C\LHnorm{\Cc_hP_h}^2H^2 (t-s)^{-1} \, \ds + \int_{t - H^2}^{t} 2\LHnorm{\Cc_hP_h}^2 \, \ds  \\
&\le CH^2 \big(\log t - 2\log H \big) + CH^2 \\
&\le C H^2 \big( \log (H^{-1}) +  t^{-1} \big),
\end{align*}
where we have used $t \le T$ for the crude estimate $\log t \le C t^{-1}$, since a $t^{-1}$-term already appears in the bounds of $R_1$ and $R_2$.
The same bound holds for $R_4$, and, by \cref{lemma:bounded_IO} and \cref{lemma:bounded_sol}, also for $R_5$ and $R_6$. In conclusion, we thus have
\begin{equation*}
  \LHnorm{\Xt(t) - \Yt(t)} \le  CH^2 \big( \log (H^{-1}) +  t^{-1} \big) + C \int_0^t  \LHnorm{\Xt(s) - \Yt(s)} \,\ds,
\end{equation*}
which by Gr\"{o}nwall's lemma yields the statement of the theorem.
\end{proof}

\begin{remark}
  In the common situation that $X(0) = 0$, corresponding to the case of no final state penalization, the $t^{-1}$-singularity disappears.
\end{remark}
\begin{remark}
  We note that a bound of the same form has been shown in~\cite{KrollerKunisch1991} for the FEM error. However, the error constant then depends on the variations in $\kappa$, and one does not observe the given convergence order until $H < \epsilon$.
\end{remark}

Similar to the parabolic case, the error bound becomes less singular near $t=0$ if we measure in the $V$-norm. To prove this we need the following, slightly stronger, assumptions on the operators (cf. \cref{ass:operators}):

\begin{assumption} \label{ass:operators2}
  In addition to \cref{ass:operators}, $\Bc \in \Lc(U, V)$, $\Cc \in \Lc(V, Z)$, and $X(0) = \Gc \in \Lc(V)$. Moreover, we assume that the mesh $\Tc_h$ is of a form such that $P_h$ is stable in the $V$-norm.
\end{assumption}
\begin{remark}
In particular, quasi-uniform meshes satisfy \cref{ass:operators2}. We refer to~\cite{Bank2014} for a discussion on more general permissible meshes.
\end{remark}

\begin{theorem}
  Suppose that \cref{ass:operators2} is fulfilled. For $t \in (0,T]$ it holds that
  \begin{equation*}
    \LVnorm{\Xt(t) - \Yt(t)} \le C H t^{-1/2}.
  \end{equation*}
  Here, the constant $C$ depends on $T$, $\alpha$, $\beta$, $\gamma$, and $\LVnorm{X(0)}$, but not on the multiscale variations of $\Ac$.
\end{theorem}

\begin{proof}
  We start by noting that $\norm{\Id_h}_{\Lc(V_h,V)}\le 1$. Furthermore, since $P_h$ is stable in the $V$-norm, the following bound holds
  \begin{align*}
    \norm{P_h}_{\Lc(V,V_h)} &= \sup_{v \in V} \frac{\norm{P_hv}_{V}}{\norm{v}_{V}} \le \sup_{v \in V} \frac{C\norm{v}_{V}}{\norm{v}_{V}} \le C.
  \end{align*}
	
  Now, note that if the initial data $v \in V$ then we may instead of \cref{lemma:parabolic_LOD} prove the following, less singular, error bound
  \begin{align*}
    \LVnorm{E_h(t) - E_\LOD(t)} \le C H t^{-1/2}.
  \end{align*}
  In addition, parabolic regularity gives the bounds $\LVnorm{E_h(t)},\LVnorm{E_\LOD(t)}\le C$.

  Note that $\norm{\Cc_h}_{\Lc(V_h, Z)} \le \norm{\Cc}_{\Lc(V, Z)}$, so from \cref{ass:operators2} it follows that
  \begin{equation}\label{operator_bound1}
    \begin{aligned}
      \LVnorm{(\Cc_hP_h)^* \Qc &(\Cc_hP_h)} \\
      &\le \norm{P_h^*}_{\Lc(V_h,V)}\norm{\Cc_h^*}_{\Lc(Z,V_h)} \norm{\Qc}_{\Lc(Z)} \norm{\Cc_h}_{\Lc(V_h,Z)}\norm{P_h}_{\Lc(V,V_h)}\le C.
    \end{aligned}
  \end{equation}
  Similarly, $\norm{\Bc_h}_{\Lc(U,V_h)} \le \norm{\Bc}_{\Lc(U,V)}$, and we have
  \begin{equation}\label{operator_bound2}
    \begin{aligned}
      \LVnorm{S_h} &\le \norm{\Id_h}_{\Lc(V_h,V)}\norm{\Bc_h}_{\Lc(U,V_h)}\norm{\Rc^{-1}}_{\Lc(U)}\norm{\Bc_h^*}_{\Lc(V_h,U)}\norm{\Id_h^*}_{\Lc(V,V_h)}\\
                   &\le C.
    \end{aligned}
  \end{equation}
	
  As in the proof of \cref{thm:main_error_L2} we can write the difference $\Xt(t)-\Yt(t)$ as a sum of eight terms so that
  \begin{align*}
    \LVnorm{\Xt(t)-\Yt(t)} \le \LVnorm{R_1} + \LVnorm{R_2} + \int_0^t \sum_{j=3}^8\LVnorm{R_j} \,\ds.
  \end{align*}
  For $R_1$ we have
  \begin{align*}
    \LVnorm{R_1} &\le \LVnorm{E_h(t)^*}\LVnorm{\Xt(0)}\LVnorm{E_h(t)-E_\LOD(t)} \\
                 &\le CHt^{-1/2}\LVnorm{X(0)} \le CHt^{-1/2},
  \end{align*}
  and similarly we prove $\LVnorm{R_2}\le CHt^{-1/2}$, where we have used that
  \begin{align*}
    \LVnorm{X_h(0)} = \LVnorm{\Id_h^* X(0) \Id_h} \le C\LVnorm{X(0)},
  \end{align*}
  which is bounded due to \cref{ass:operators2}. Using the bounds \cref{operator_bound1,operator_bound2} we get
  \begin{align*}
    \int_0^t\sum_{j=3}^6\LVnorm{R_j} \,\ds \le \int_0^tCH(t-s)^{-1/2} \,\ds \le CHt^{1/2},
  \end{align*}
  and
  \begin{align*}
    \int_0^t\LVnorm{R_7} + \LVnorm{R_8} \,\ds \le \int_0^tC\LVnorm{\Xt(s) - \Yt(s)} \,\ds.
  \end{align*}
  By applying Gr\"{o}nwall's lemma we obtain the desired error bound.
\end{proof}

\section{Matrix-valued formulation} \label{sec:matrices}
To perform actual computations, we write the finite-dimensional equations on matrix form by expressing the equations in the FEM or LOD bases.
To this end, let the function $x \in V_h$ and the operator $X_h \colon V_h \to V_h$ have the vector and matrix representations $\xb \in \R^{N_h}$ and $\tilde{\Xb}^h \in \R^{N_h \times N_h}$, i.e.
\begin{equation}\label{eq:Xh_formula1}
  x = \sum_{j=1}^{N_h}{\xb_j \phi^h_j} \quad \text{and } \quad X_h x = \sum_{i,j=1}^{N_h}{ \tilde{\Xb}^h_{i,j} \xb_j \phi^h_i}
\end{equation}
Since exactly the same results hold for $V_H$ upon replacing $h$ by $H$, we frequently omit the $h$ sub- and superscripts in the following manipulations. They will be reinstated later when we compare different discretizations.
 The coordinates satisfy
\begin{equation*}
  \Mb \xb = \big(\iprod{x}{\phi_i} \big)_{i=1}^{N} \quad \text{and } \quad \Mb \tilde{\Xb} = \big(\iprod{X_h \phi_j}{\phi_i} \big)_{i,j=1}^{N},
\end{equation*}
 where $\Mb$ denotes the (symmetric) mass matrix, $\Mb_{i,j} = \iprod{\phi_j}{\phi_i}$. Unfortunately, we will not recover the usual form of the matrix-valued DRE when working in these coordinates. Therefore, we perform the change of variables
 \begin{equation*}
  \Xb \Mb = \tilde{\Xb}.
 \end{equation*}
Coincidentally, this means that we actually have
\begin{equation} \label{eq:Xh_formula2}
  X_h x = \sum_{i,j=1}^N{\Xb_{i,j}\iprod{x}{\phi_j} \phi_i}.
\end{equation}

\Cref{eq:DRE_FEM} is equivalent to
\begin{align} \label{eq:DRE_FEM_basis}
  \iprod{\dot{X}_h \phi_i}{\phi_j} &= \iprod{X_h \phi_i}{\Ac_h  \phi_j} + \iprod{X_h \phi_j}{\Ac_h \phi_i} \\
&+ \iprod{\Qc\Cc_h \phi_i}{\Cc_h \phi_j}_Z - \iprod{\Rc^{-1} \Bc_h^* X_h \phi_i}{\Bc_h^* X_h \phi_j}_U   \nonumber
\end{align}
for $1 \le i,j \le N$, and since $X_h \phi_i = \sum_{k=1}^N{(\Xb \Mb)_{k,i} \phi_k}$, the first term becomes
\begin{equation*}
\sum_{k = 1}^{N}{(\dot{\Xb} \Mb)_{k,i}} \Mb_{j,k} = (\Mb \dot{\Xb} \Mb)_{j,i}
\end{equation*}
Likewise, with the (negative) stiffness matrix $\Ab_{i,j} = \iprod{\Ac\phi_j}{\phi_i}$, the second and third terms become
\begin{equation*}
\sum_{k=1}^N{(\Xb\Mb)_{k,i} \Ab_{k,j}} + \sum_{k=1}^N{(\Xb\Mb)_{k,j} \Ab_{k,i}}  = (\Ab^T\Xb\Mb)_{j,i} + (\Mb\Xb\Ab)_{j,i},
\end{equation*}
due to the symmetry of $\Mb$ and $\Xb$. (Recall that we search for a self-adjoint operator $X_h$.)
Finally, the last two terms can be written
\begin{equation*}
    (\Cb^T \Qb \Cb)_{j,i} \quad \text{and } \quad \big( \Mb \Xb \Bb \Rb^{-1} \Bb^T  \Xb \Mb \big)_{j,i},
\end{equation*}
where $\Bb_{i,j} = \iprod{\Bc\phi^U_j}{\phi_i}$, $\Qb_{i,j} = \iprod{\Qc\phi^Z_j}{\phi^Z_i}$, $\Rb_{i,j} = \iprod{\Rc\phi^U_j}{\phi^U_i}$, $\Cb_{i,j} = \iprod{\Cc\phi_j}{\phi^Z_i}$ and $\{\phi^U_i\}$, $\{\phi^Z_i\}$ denote orthonormal bases for $U$ and $Z$, respectively. Summarizing, we can write the equation on matrix form as
\begin{equation} \label{eq:DRE_matrix}
\Mb \dot{\Xb} \Mb = \Mb \Xb \Ab + \Ab^T \Xb \Mb + \Cb^T \Qb \Cb - \Mb \Xb \Bb \Rb^{-1} \Bb^T \Xb \Mb.
\end{equation}

Similar to the relations between the fine and coarse operators \cref{eq:fine_coarse_operator_relations}, it is easily shown that their matrix representations satisfy
\begin{equation*}
  \Ab_H = (\IbHh)^T \Ab_h \IbHh, \quad \Bb_H = (\IbHh)^T \Bb_h, \quad \Cb_H = \Cb_h \IbHh \quad \text{and} \quad \Mb_H = (\IbHh)^T \Mb_h \IbHh,
\end{equation*}
where $\IbHh \in \R^{N_h \times N_H}$ is the prolongation matrix that satisfies $\IbHh \xb^H = \xb^h$ if $x = \sum_{j=1}^{N_H}{\xb^H_j \phi^H_j}$ and $\IdHh x = \sum_{j=1}^{N_h}{\xb^h_j \phi^h_j}$. By expressing the $\phi^H$ functions in terms of $\phi^h$, it can be seen that $(\IbHh)_{i,j} = \phi^H_j(z_i)$, where $z_i$ is the $i$:th node of $\Tc_h$. Thus the coarse systems are easily constructed when the fine system is known. Note, however, that the matrix representation of $(\IdHh)^*$ is not $(\IbHh)^T$ but  $\Mb_H^{-1} (\IbHh)^T \Mb_h$.

For the LOD case, we let $\Qb_h$ and $\Rb_h = \IbHh - \Qb_h$ be the matrix representations of $Q_h$ and $R_h$, respectively. To compute them efficiently, we follow~\cite{Engwer_etal_2016}. Then
\begin{equation*}
  \XLOD_H x = \sum_{i=1}^{N_H}{ (\Xb^\LOD_H \Mb_\LOD \xb)_i \phi^H_i},
\end{equation*}
where $\Xb^\LOD_H$ is symmetric and satisfies
\begin{align*}
\Mb_\LOD \dot{\Xb}^\LOD_H \Mb_\LOD &= \Mb_\LOD \Xb^\LOD_H \Ab_\LOD + \Ab_\LOD^T \Xb^\LOD_H \Mb_\LOD \\
&+ \Cb_\LOD^T \Qb \Cb_\LOD - \Mb_\LOD \Xb^\LOD_H \Bb_\LOD \Rb^{-1} \Bb_\LOD^T \Xb^\LOD_H \Mb_\LOD ,
\end{align*}
with the matrices
\begin{equation*}
  \Ab_\LOD = \Rb_h^T \Ab_h \Rb_h, \quad \Bb_\LOD = \Rb_h^T \Bb_h, \quad \Cb_\LOD = \Cb_h \Rb_h \quad \text{and} \quad \Mb_\LOD = \Rb_h^T \Mb_h \Rb_h.
\end{equation*}
Finally, we note that if $u \in V_h$, $w \in V_H$ and $(\IdLODh)^* u = R_h w$,  then in coordinates we have $\bm{w} = \Mb_\LOD^{-1} \Rb_h^T \Mb_h \bm{u}$. This means that the matrix representation of $\IdLODh \XLOD_h (\IdLODh)^*$ is $\Rb_h \Xb^\LOD_H \Rb_h^T \Mb_h$.

\subsection{Error computation} \label{subsec:matrix_errors}
We measure the quality of different approximations as the $\LHH$-normed distance to a reference approximation at the final time $T$. In order to find a matrix representation for this, we first observe that since $\norm{P_h x} \le \norm{x}$, we have
\begin{equation*}
  \LHnorm{\Id_h X_h P_h} = \sup_{\substack{x \in \Ltwo \\ x \neq 0}} \frac{\norm{X_h P_hx}}{\norm{x}} \le \sup_{\substack{x \in \Ltwo \\ x \neq 0}} \frac{\norm{X_h P_hx}}{\norm{P_hx}} = \sup_{\substack{x \in V_h \\ x \neq 0}} \frac{\norm{X_h x}}{\norm{x}}.
\end{equation*}
But $P_h x = x$ for $x \in V_h$, so since $V_h \subset \Ltwo$ we also get
\begin{equation*}
  \LHnorm{\Id_h X_h P_h} \ge \sup_{\substack{x \in V_h \\ x \neq 0}} \frac{\norm{X_h P_hx}}{\norm{x}} = \sup_{\substack{x \in V_h \\ x \neq 0}} \frac{\norm{X_h x}}{\norm{x}}.
\end{equation*}
To compute the $\LHH$-norm it is thus enough to test with $x  = \sum_{i=1}^{N_h}{\xb_i \phi^h_i} \in V_h$. Again omitting the $h$ sub- and superscripts, we have that
$\iprod{x}{x} = \xb^T \Mb \xb$, and similarly
\begin{align*}
  \iprod{X_h x}{X_h x} &= \sum_{i,j,k,l = 1}^N{(\Xb\Mb)_{i,j}\xb_j(\Xb\Mb)_{k,l}\xb_l\iprod{\phi_i}{\phi_k}}\\
 &= \xb^T \Mb^T \Xb^T \Mb \Xb \Mb \xb.
\end{align*}
Since $\Mb$ is symmetric positive definite, we may do a Cholesky factorization $\Mb = \Lb_{\Mb} \Lb_{\Mb}^T$, and the change of variables $\yb = \Lb_{\Mb}^T \xb$ yields
\begin{equation*}
  \LHnorm{\Id_h X_h P_h} = \sup_{\substack{\yb \in \R^N \\ \yb \neq 0}} \frac{\big(\yb^T \Lb_{\Mb}^T \Xb \Lb_{\Mb} \Lb_{\Mb}^T \Xb \Lb_{\Mb} \yb \big)^{1/2}}{\big(\yb^T \yb\big)^{1/2}} = \norm{\Lb_{\Mb}^T \Xb \Lb_{\Mb}}_{\R^{N \times N}},
\end{equation*}
where $\norm{\cdot}_{\R^{N \times N}}$ denotes the standard spectral matrix norm.
Recalling the matrix representation $\IbHh$ of $\IdHh$, we now get that
\begin{equation*}
  \LHnorm{\Id_h X_h P_h - \Id_H X_H P_H} = \norm{\Lb_{\Mb}^T (\Xb_h - \IbHh \Xb_H (\Ib_H^h)^T) \Lb_{\Mb}}_{\R^{N \times N}}.
\end{equation*}
The LOD error is completely analogous, using instead $\Rb_h$ and $\Xb^\LOD_H$.

A similar approach also allows us to compute $\LVV$-errors. Let $\Ab = \Lb_{\Ab} \Lb_{\Ab}^T$ be a Cholesky factorization of $\Ab$. Then
\begin{equation*}
  \LVnorm{\Id_h X_h P_h - \Id_H X_H P_H} \le \norm{P_h}_{\Lc(V,V_h)} \norm{\Lb_{\Ab}^T (\Xb_h - \IbHh \Xb_H (\Ib_H^h)^T) \Mb \Lb_{\Ab}^{-T}}_{\R^{N \times N}}.
\end{equation*}
We also get that $\LVnorm{\Id_h X_h P_h - \Id_H X_H P_H}$ is bounded from below by $\norm{\Lb_{\Ab}^T (\Xb_h - \IbHh \Xb_H (\Ib_H^h)^T) \Mb \Lb_{\Ab}^{-T}}_{\R^{N \times N}}$, i.e.\ the latter quantity can be thought of as an equivalent norm.
Since $\Lb_{\Ab}$ is triangular, the extra cost required for the computation of $\Lb_{\Ab}^{-T}$ is negligible. If the low-rank formulation is used (see \cref{subsec:lowrank_errors}), only a small number of linear equation systems involving $\Lb_{\Ab}$ needs to be solved, reducing the cost even further.

\section{Temporal discretization} \label{sec:time}
We discretize the matrix-valued DREs in time by means of a low-rank splitting scheme, since the basic operation in such methods is the application of $\exp{t\Ab^T}$, i.e.\ essentially solving a parabolic problem. Let $\tau$ denote a fixed time step, and let $t_j = j\tau$, $j = 0, \ldots, N_t$, be the time discretization of the interval $[0, T]$. We split Equation \cref{eq:DRE_matrix} into two parts, $\dot{\Xb} = \Fs \Xb + \Gs \Xb$, where
\begin{equation*}
  \Fs \Xb = \Xb \Ab \Mb^{-1} + \Mb^{-1} \Ab^T \Xb + \Mb^{-1}\Cb^T \Qb \Cb \Mb^{-1} \quad \text{and} \quad \Gs \Xb = -\Xb \Bb \Rb^{-1} \Bb^T \Xb.
\end{equation*}
Then the Strang splitting approximation at time $t_j$ is given by $\Xb^j$, with $\Xb^0 = \Xb(0)$ and
\begin{equation*}
  \Xb^{j+1} = \exp{\frac{\tau}{2}\Fs} \exp{\tau\Gs} \exp{\frac{\tau}{2}\Fs} \Xb^j.
\end{equation*}
Here, the solution operators $\exp{t\Fs}$ and $\exp{t\Gs}$ satisfy
\begin{align}
\label{eq:split_etF}
  \exp{t\Fs} \Xb &= \exp{t\Mb^{-T}\Ab^T} \Xb \exp{t\Ab\Mb^{-1}} + \int_0^t{\exp{s\Mb^{-T}\Ab^T} \Mb^{-T} \Cb^T \Qb \Cb \Mb^{-1} \exp{s\Ab\Mb^{-1}} \,\ds}, \\
\label{eq:split_etG}
\exp{t\Gs} \Xb &= \big( I + t\Xb \Bb \Rb^{-1} \Bb^T \big)^{-1} \Xb,
\end{align}
where the first equality is apparent from the integral formulation \cref{eq:DRE_integral}, while the second is easily verified by differentiation.

The low-rank version of the method relies on the assumption that the solution $\Xb$ has low rank. This is general true for LQR problems and dramatically reduces the  computational cost. In that case, we may factorize $\Xb = \Lb \Db \Lb^T$, where $\Lb \in \R^{N_h \times r}$ and $\Db \in \R^{r \times r}$ with the rank $r \ll N_h$. Also $\exp{\tau\Fs}\Xb$ and $\exp{\tau\Gs}\Xb$, and thus the iterates $\Xb_j$, may then be factorized in such a way. After a reformulation, $\exp{\tau\Gs}\Xb$ is very cheap to compute, and the computation of $\exp{\tau\Fs}\Xb$ reduces to an evaluation of $\exp{\tau\Mb^{-T}\Ab^T} \Lb$ (plus preliminary, similar work for the integral term). The latter operation is equivalent to solving $\Mb \dot{x} = \Ab^T x $, $x(0) = \Lb$, and the matrix $\Mb$ is thus never explicitly inverted. For further details, we refer to~\cite{Stillfjord2015, Stillfjord2017}.

\subsection{Low-rank errors} \label{subsec:lowrank_errors}
Also the error computations outlined in \cref{subsec:matrix_errors} benefit from being formulated in a low-rank setting. Assume that $\Xb_h = \Lb_h\Db_h\Lb_h^T$ and $\Xb_H = \Lb_H\Db_H\Lb_H^T$ with $\Lb_h \in \R^{N_h \times r_h}$ and $\Lb_H \in \R^{N_H \times r_H}$ with $r_h, r_H \ll N_h$, and let $\Mb_h = \Lb_{\Mb} \Lb_{\Mb}^T$ be a Cholesky factorization. By setting
\begin{equation*}
  \Vb =
  \begin{bmatrix}
   \Lb_{\Mb}^T \Lb_h & \Lb_{\Mb}^T \IbHh \Lb_H
  \end{bmatrix}
  \quad \text{and} \quad
  \Db =
  \begin{bmatrix}
    \Db_{h}  &     0 \\
       0    & -\Db_{H}
  \end{bmatrix}
\end{equation*}
we see that $\Vb \in \R^{N_h \times (r_h + r_H)}$, $\Db \in \R^{(r_h + r_H) \times (r_h + r_H)}$ and it follows that
\begin{equation*}
  \Lb_{\Mb}^T (\Xb_h - \IbHh \Xb_H (\Ib_H^h)^T) \Lb_{\Mb} = \Vb \Db \Vb^T.
\end{equation*}
Since $\Vb\Db\Vb^T$ is not necessarily an eigenvalue decomposition, we cannot immediately determine the norm by inspection. However, performing a QR-factorization $\Vb = \Qb \Rb$ is cheap if the number of columns is low, and $\Rb\Db\Rb^T \in \R^{(r_h + r_H) \times (r_h + r_H)}$ can also be diagonalized cheaply. (This is precisely the $LDL^T$ column compression procedure which is applied in each time step.) We acquire $\Vb\Db\Vb^T = (\Qb\Wb)\tilde{\Db}(\Qb\Wb)^T$, for some $\Wb$, where $\norm{\Vb\Db\Vb^T} = |\tilde{\Db}_{1,1}|$.

For errors in the $\LVV$-norm, we do not get a symmetric matrix as above. But if $\Ab = \Lb_{\Ab}\Lb_{\Ab}^T$ we can still write
\begin{equation*}
  \Lb_{\Ab}^T (\Xb_h - \IbHh \Xb_H (\Ib_H^h)^T) \Mb \Lb_{\Ab}^{-T} = \Gb_1 \Db \Gb_2^T,
\end{equation*}
with the same $\Db$, and with
\begin{equation*}
  \Gb_1 =
  \begin{bmatrix}
   \Lb_{\Ab}^T \Lb_h & \Lb_{\Ab}^T \IbHh \Lb_H
  \end{bmatrix}
  \quad \text{and} \quad
  \Gb_2 =
  \begin{bmatrix}
   \Lb_{\Ab}^{-1} \Mb \Lb_h & \Lb_{\Ab}^{-1} \Mb \IbHh \Lb_H .
  \end{bmatrix}
\end{equation*}
We can cheaply $QR$-factorize both $\Gb_1 = \Ub\Rb_1$ and $\Gb_2 = \Vb\Rb_2$; this means that
\begin{equation*}
 \norm{\Lb_{\Ab}^T (\Xb_h - \IbHh \Xb_H (\Ib_H^h)^T) \Mb \Lb_{\Ab}^{-T}}_{\R^{N_h \times N_h}} = \norm{\Ub \Sb \Vb^T}_{\R^{N_h \times N_h}} = \norm{\Sb}_{\R^{N_h \times N_h}},
\end{equation*}
where $\Sb = \Rb_1 \Db \Rb_2^T$ is a small matrix.

 \section{Numerical experiments} \label{sec:experiments}
We have performed a number of numerical experiments in order to verify our a priori error bounds for the LOD discretizations, and to demonstrate their efficiency in comparison to the classical FEM.

In all experiments, we compute the relevant matrices for both FEM and LOD by using efficient code written by Fredrik Hellman and Daniel Elfverson\footnote{Available on request from Fredrik Hellman, \email{fredrik.hellman@it.uu.se}.}. These pre-solve computations were run on a Intel\textsuperscript{\textregistered} Core\texttrademark{} i5-4690 processor. We note that the localized elliptic fine-scale problems were not solved in parallel. Doing so would further improve the performance of LOD.

For approximating the solutions to the DREs, we employ in all cases the low-rank Strang splitting scheme (as described in \cref{sec:time}) with $N_t = 256$ time steps. This ensures that the temporal error is small compared to the spatial error, which is our interest here. Our implementation utilizes the DREsplit\footnote{Available on request from Tony Stillfjord, \email{tony.stillfjord@gu.se}, or from \url{http://www.tonystillfjord.net}.} library. These computations were performed on resources at Chalmers Centre for Computational Science and Engineering (C3SE) provided by the Swedish National Infrastructure for Computing (SNIC). Each simulation used a single Intel\textsuperscript{\textregistered} Xeon\textsuperscript{\textregistered} E5-2650 v3 processor.

The multiscale diffusion coefficients $\kappa$ considered in the numerical examples are of two distinct types. In Examples 1, 2, and 4 we consider a piecewise constant coefficient, generated randomly with no spatial correlation, that varies on a fine scale, see \cref{fig:diff_coeff_grid}. In Examples 3 and 5 $\kappa$ takes two values. One value in the background and one in the thin channels, see \cref{fig:diff_coeff_Ushape}. This is a common setup for reinforced (composite) materials. Both these cases are challenging for the finite element method.

\subsection{Example~1}\label{example:grid}
In this first example, we consider diffusion on the unit square. More specifically, we take $\Omega = [0,1]^2$ and set $\Ac x = \nabla \cdot \big( \kappa \nabla x \big)$ with Dirichlet boundary conditions. Here, $\kappa$ is piecewise constant on a square grid of size $2^{-7}$ and taking randomly chosen values in $[10^{-3}, 1]$; see \cref{fig:diff_coeff_grid} for an illustration. We consider $3$ independent inputs and define the input operator $\Bc$ as the sum
\begin{equation*}
  \Bc u = \sum_{j=1}^{3} \Bc_j u_j,
  \quad \text{where} \quad
  (\Bc_j u)(\xi_1, \xi_2) =
  \begin{cases}
    u,  &  \frac{j}{4} \le \xi_1, \xi_2 \le \frac{j}{4} + \frac{1}{8}\\
    0, & \text{otherwise}
  \end{cases}
.
\end{equation*}
Thus we can control the system on three small squares. As the output operator we take the mean, i.e.\ $\Cc x = \int_{\Omega}{x}$. We choose $\Qc$ and $\Rc$ to be the identity operators and take $\Gc = X(0) = 0$.
\begin{figure}[t!]
\centering
\includegraphics[width=0.6\columnwidth]{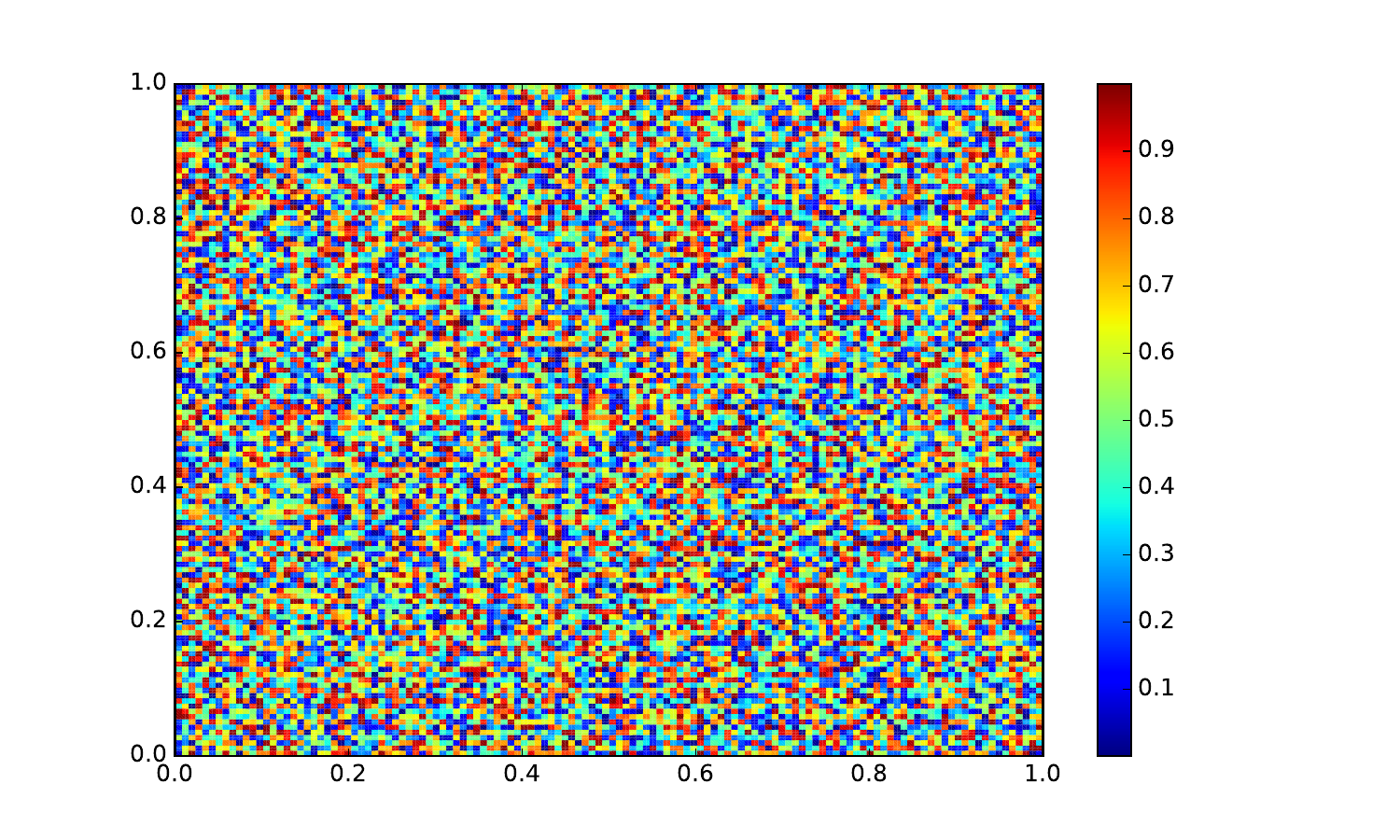}
\caption{The diffusion coefficient used in \nameref{example:grid}, plotted over the domain $\Omega$. (This figure is in color in the electronic version of the article.)}
\label{fig:diff_coeff_grid}
\end{figure}

For the discretization in space, we start with a coarse mesh containing $8$ triangles, and then refine this $6$ times, giving meshes with $2^{3 + 2j}$ triangles, for ${j = 0, \ldots, 6}$. One additional refinement provides the reference grid with $2^{17} = 131072$ triangles.
This results in matrices $\Ab_j \in \R^{n \times n}$, $\Bb_j \in \R^{n \times 3}$ and $\Cb_j \in \R^{1 \times n}$, $j = 0, \ldots, 7$, with $n = 1, 9, 49, 225, 961, 3969, 16129, 65025$ (since we only consider the interior nodes).

The approximations are compared only at the final time, in the $\LHH$- and $\LVV$-norms as outlined in \cref{subsec:matrix_errors}, and the computed errors are shown in \cref{fig:grid}. We see that the classical FEM initially struggles due to not resolving the multiscale coefficient properly, but converges with order $2$ when the mesh becomes fine enough. The LOD approach converges with order $2$ also for the coarse meshes, and additionally results in approximations that are about one order of magnitude more accurate.
The plot to the right shows the errors against the actual computation time, including the time spent on constructing the LOD bases. As can be seen, this extra effort is low enough that except for the most inaccurate cases it is always worthwhile to use the LOD approach.

\begin{figure}[t!]
\centering
\includegraphics[width=0.49\columnwidth]{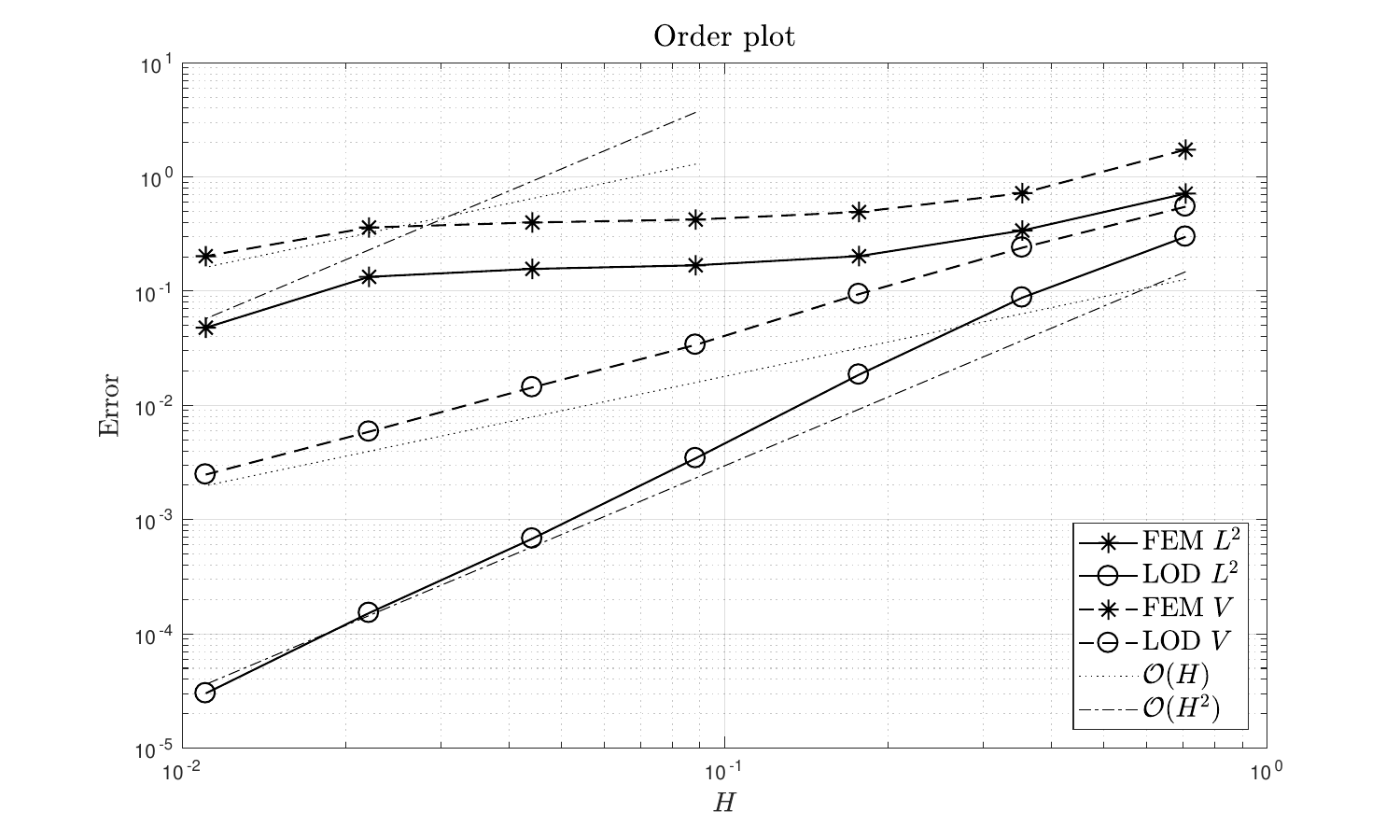}
\includegraphics[width=0.49\columnwidth]{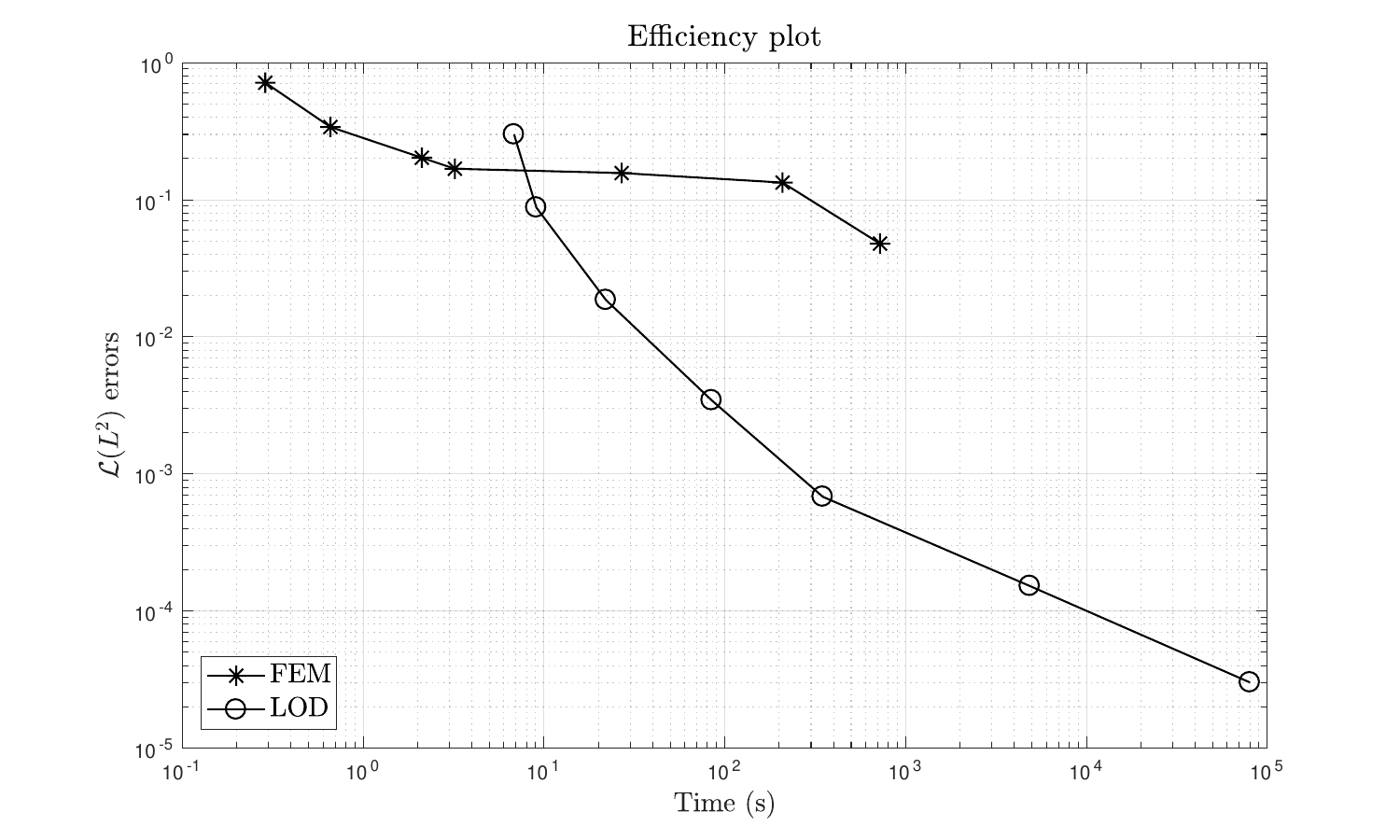}
\caption{Left: The $\LHH$- and $\LVV$-norm errors of the approximations computed in \nameref{example:grid}, plotted against the meshwidth. Right: The $\LHH$-norm errors plotted against the computation time.}
\label{fig:grid}
\end{figure}

\subsection{Example~2}\label{example:Lshape}
Here, we consider an L-shaped domain $\Omega$, where $[0.5, 1] \times [0.5, 0.5]$ has been removed from the unit square. The diffusion coefficient $\kappa$ is piecewise constant on a square grid of size $2^{-7}$ and taking random values in $[10^{-3}, 1]$.
 We use one control input, given by the characteristic function of the square $[0.65, 0.85]^2$, and one output, the mean over the square $[0.15, 0.35]^2$. The meshes are setup as in the previous example, but now with $n = 5, 33, 161, 705, 2945, 12033, 48641$ interior nodes ($n=195585$ for the reference solution). The time discretization and other parameters are the same as in the previous example.

The results are shown in \cref{fig:Lshape}. Due to the reentrant corner the errors behave more erratically than in the previous example, but LOD is still clearly first- and second-order convergent in contrast to standard FEM, which performs very poorly. We also observe that LOD is more efficient in all but the coarsest cases.

\begin{figure}[t!]
\centering
\includegraphics[width=0.49\columnwidth]{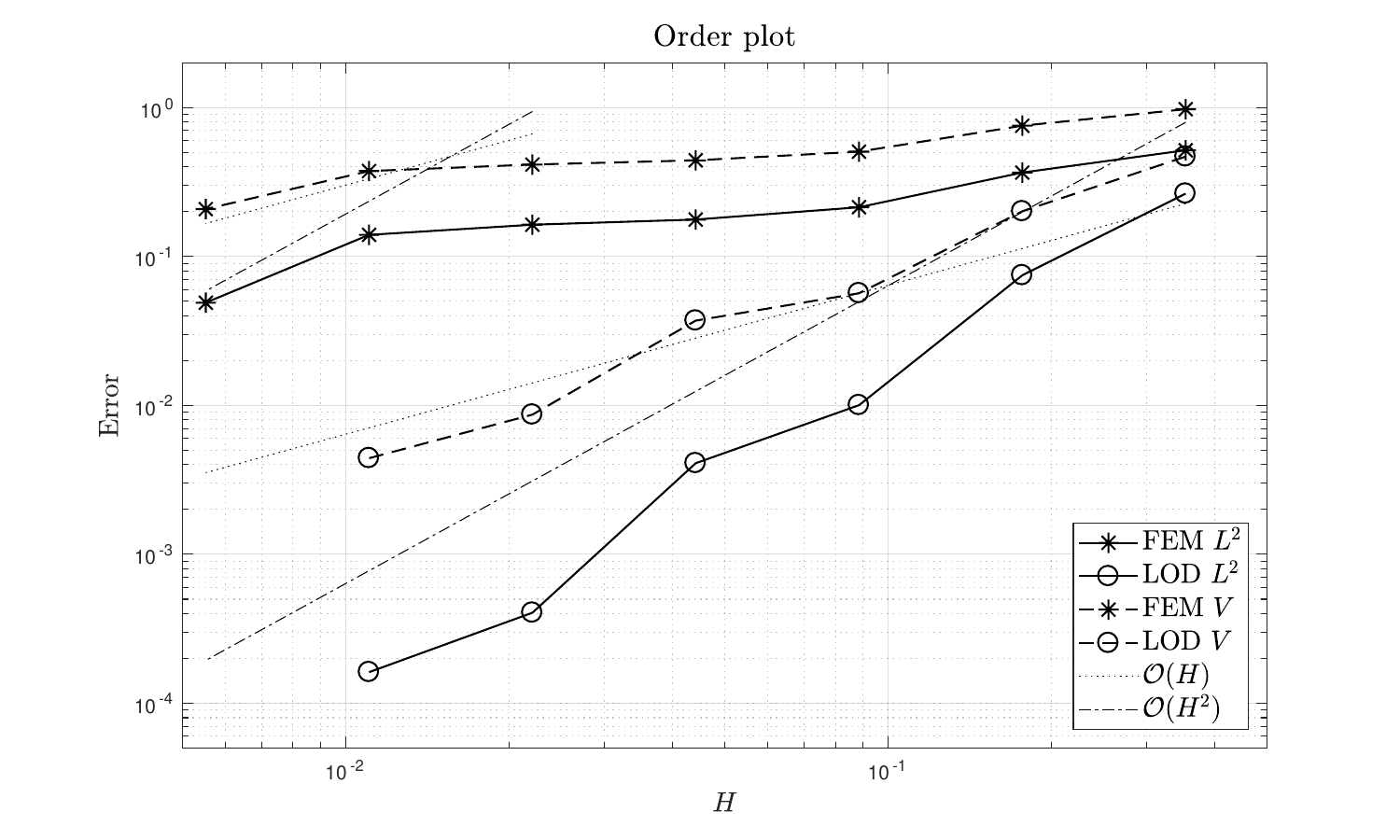}
\includegraphics[width=0.49\columnwidth]{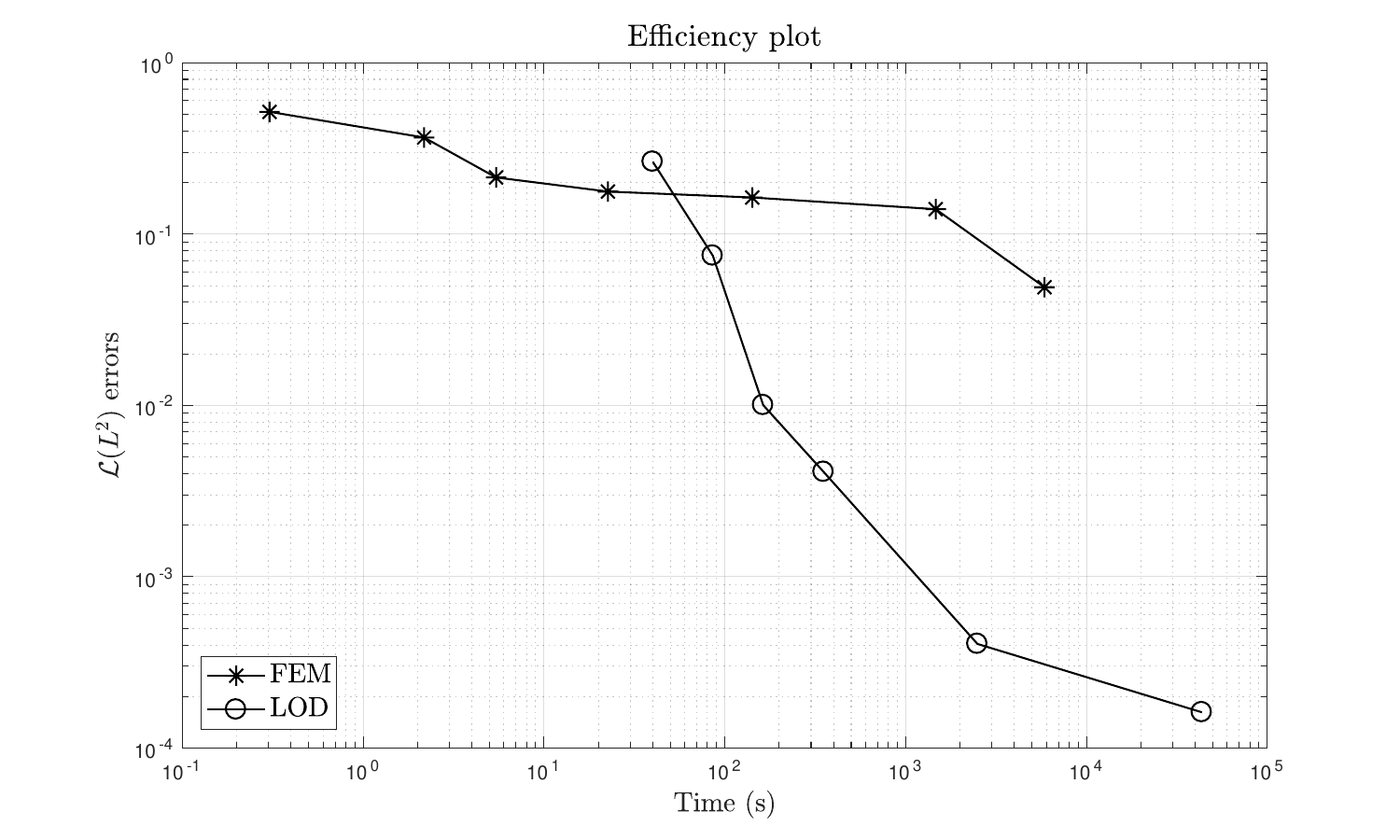}
\caption{Left: The $\LHH$- and $\LVV$-norm errors of the approximations computed in \nameref{example:Lshape}, plotted against the meshwidth. Right: The $\LHH$-norm errors plotted against the computation time.}
\label{fig:Lshape}
\end{figure}

\subsection{Example~3}\label{example:stripes}
We again consider the setting of \nameref{example:grid}, but replace the diffusivity constant. Here, $\kappa$ takes the constant value $1$ everywhere, except for in $7$ horizontal stripes where it is $10^{-2}$. The stripes are centered around the heights $j/8$, $j=1,\ldots,7$, and have a width of $2^{-7}$.

The results are shown in \cref{fig:stripes}. This time, the detrimental effect on the FEM discretization is even more pronounced, with almost no convergence until the thin stripes can be resolved. The LOD approximations are once again more accurate for all $H$. We note that the $\LHH$-error is not quite $\mathcal{O}(H^2)$ in this case, but rather close to $\mathcal{O}(H^2\log{H^{-1}})$ as predicted by \cref{thm:main_error_L2}. Like in the previous example, computing the LOD bases is cheap enough that the LOD approach is more efficient in all but the least accurate cases.

\begin{figure}[t!]
\centering
\includegraphics[width=0.49\columnwidth]{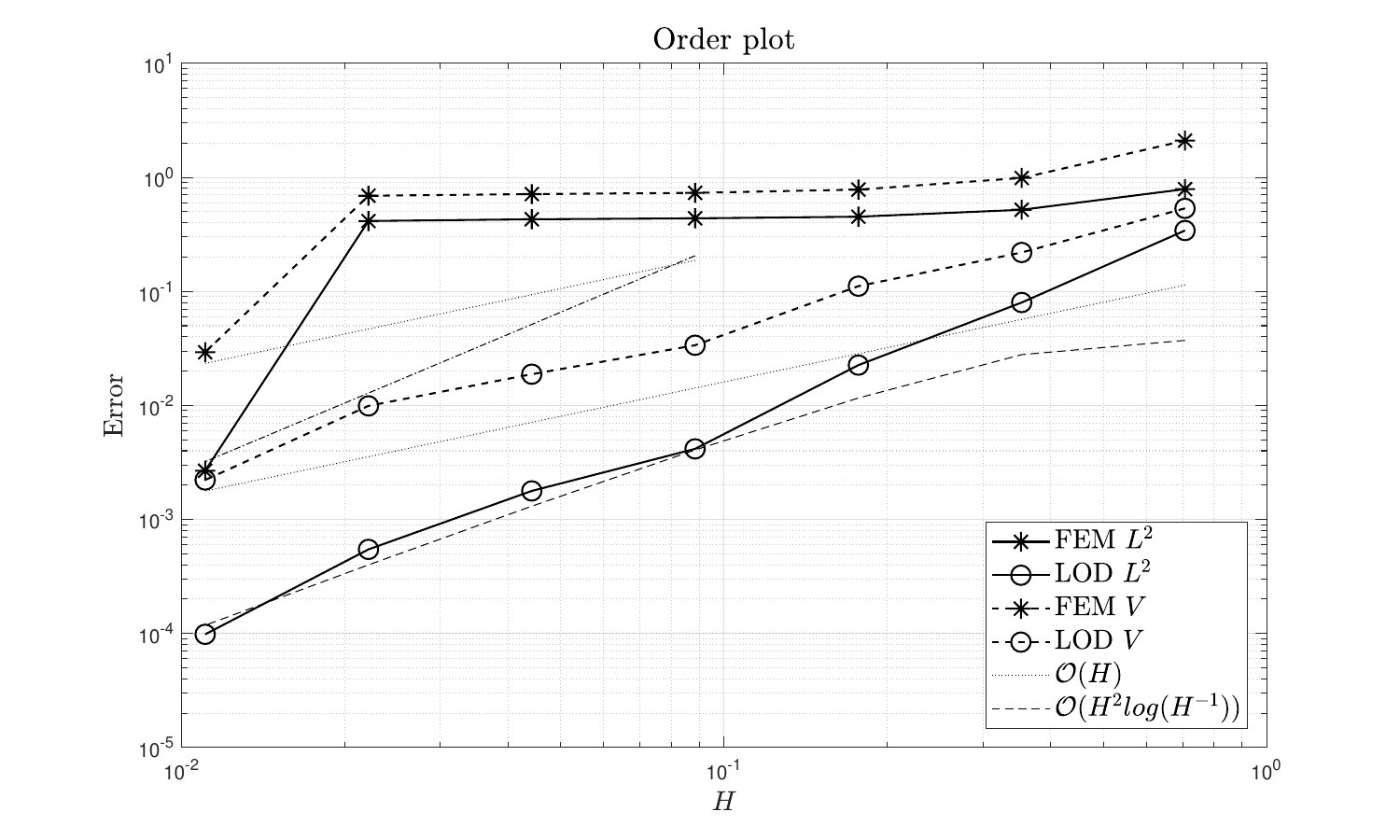}
\includegraphics[width=0.49\columnwidth]{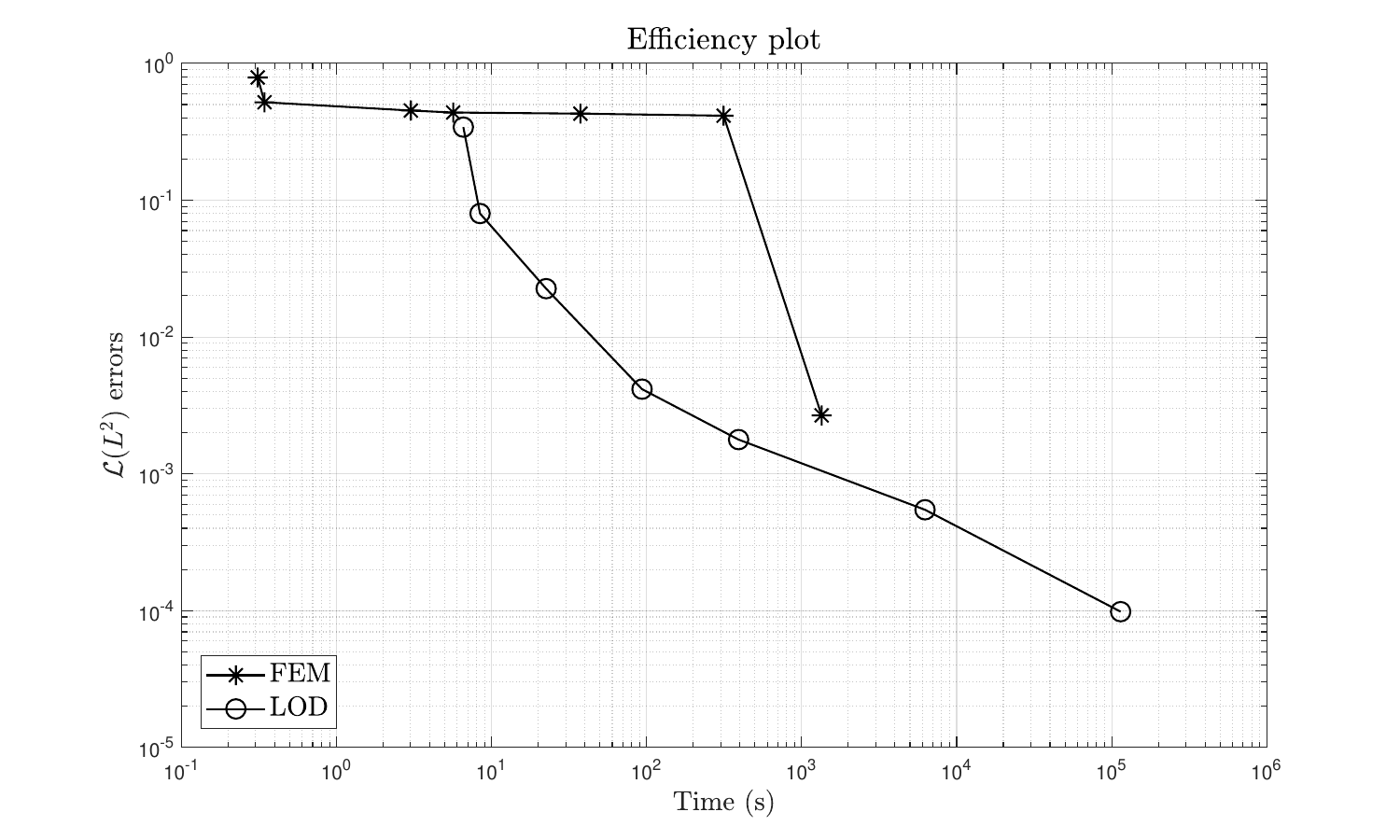}
\caption{Left: The $\LHH$- and $\LVV$-norm errors of the approximations computed in \nameref{example:stripes}, plotted against the meshwidth. Right: The $\LHH$-norm errors plotted against the computation time.}
\label{fig:stripes}
\end{figure}

\subsection{Example~4}\label{example:BCC_grid}
In this example, we deviate from the basic setting described in \cref{sec:errors} by considering a boundary control application. All parameters except for the boundary conditions and the input operator are the same as in \nameref{example:grid}.
  We call the union of the top and bottom edges of the unit square $\Gamma_D$ and impose homogeneous Dirichlet boundary conditions there. The left and right edges we denote $\Gamma_1$ and $\Gamma_2$, respectively, and there we impose nonhomogeneous Neumann boundary conditions. In particular, with the outward-pointing normal denoted by $n$, we consider functions $x$ satisfying
\begin{equation*}
  \kappa \nabla x \cdot n = \Psi u_i \quad \text{on } \Gamma_i.
\end{equation*}
Here, $u_1$ and $u_2$ are the two control inputs, and
\begin{equation*}
\Psi \colon s \mapsto
\begin{cases}
  2s, & 0 \le s \le 1/2, \\
2(1-s), & 1/2 < s \le 1,
\end{cases}
\end{equation*}
is a fixed function. The operator $\Ac$ now corresponds to $x \mapsto \nabla \cdot \big(\kappa \nabla x \big)$ on the space $\{x \in H^1(\Omega) \;|\; x_{\rvert_{\Gamma_D}} = 0 \}$  with no conditions imposed on $\Gamma_1$, $\Gamma_2$, while the (unbounded) operator $\Bc$ implements the Neumann boundary conditions. We refrain from elaborating further on this here, and simply note that the FEM matrix representation becomes
\begin{equation*}
  \Bb_{j,i}^h = \int_{\Gamma_i} \Psi \phi_j^h.
\end{equation*}
Fur further details on the proper abstract framework, see e.g.~\cite{LasieckaTriggiani2000} and \cref{subsec:BCC}.

Since \cref{ass:operators} is no longer satisfied, we may not apply \cref{thm:main_error_L2}. However, the results plotted in \cref{fig:BCC_grid} are similar to the results in previous examples. Again, the LOD approximations are more efficient except for the very coarsest meshes. This indicates that our theory could be extended also to the case of unbounded operators $\Bc$ and $\Cc$.

\begin{figure}[t!]
\centering
\includegraphics[width=0.49\columnwidth]{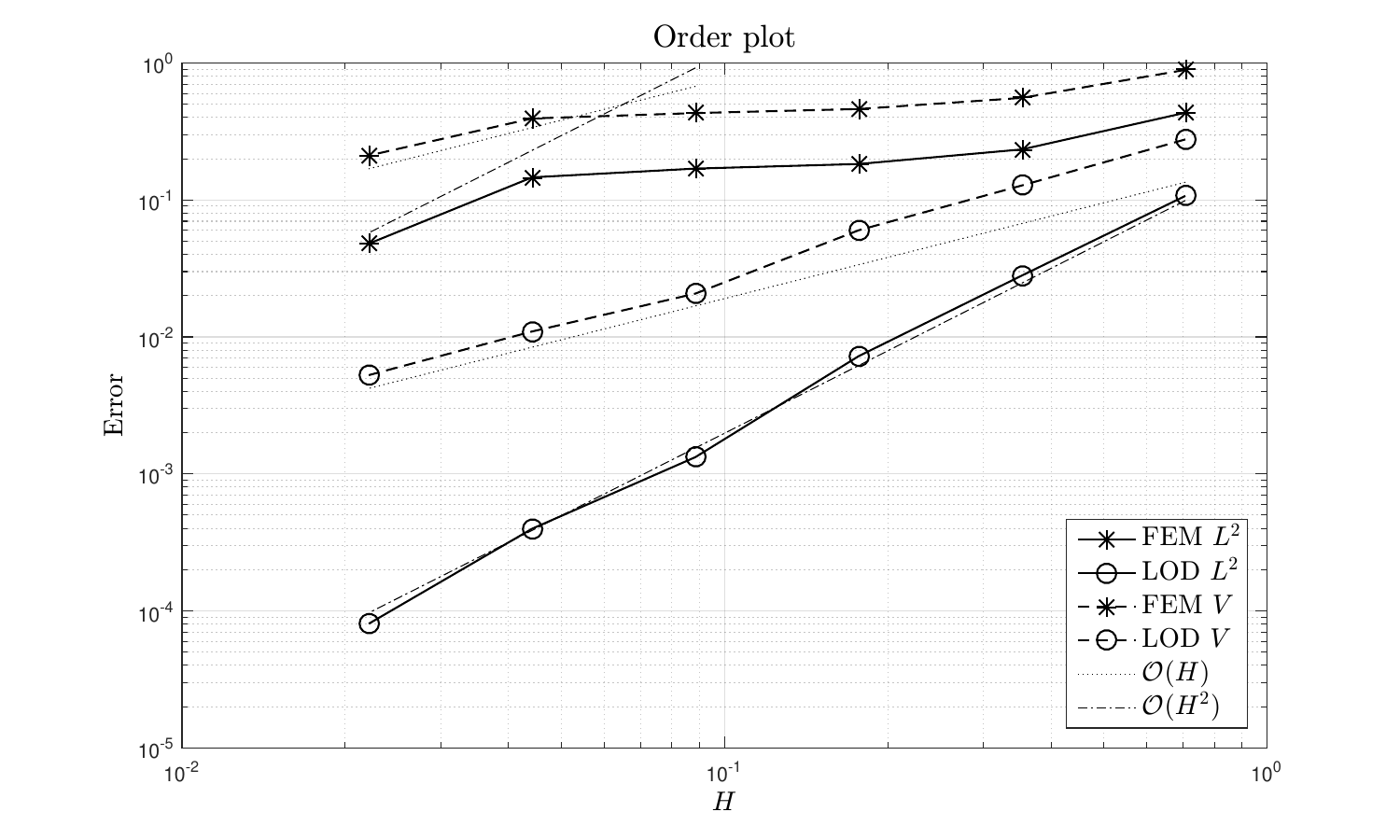}
\includegraphics[width=0.49\columnwidth]{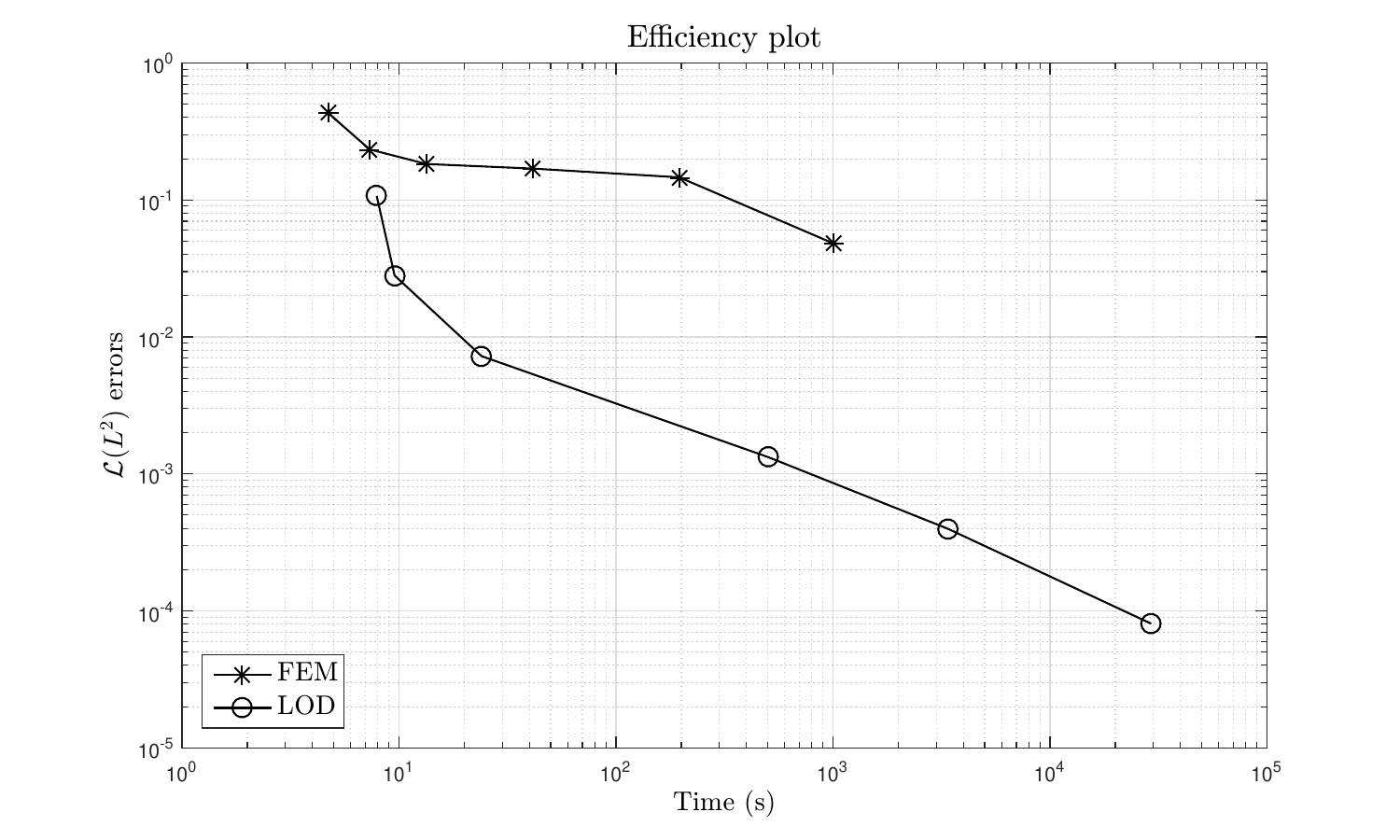}
\caption{Left: The $\LHH$- and $\LVV$-norm errors of the approximations computed in \nameref{example:BCC_grid}, plotted against the meshwidth. Right: The $\LHH$-norm errors plotted against the computation time.}
\label{fig:BCC_grid}
\end{figure}

\subsection{Example~5}\label{example:BCC_Ushape}
As a final experiment, we consider another boundary control application. The domain is formed like a lying U, see \cref{fig:diff_coeff_Ushape}. The thickness of each of the ``handles'' is $1/6$, the total horizontal extent $1$ and the vertical extent $4/6$. Inside the domain are three evenly spaced stripes with a diameter of $0.0052$. As previously, we consider $\Ac u = \nabla \cdot \big( \kappa \nabla u \big)$ where $\kappa = 10^{-2}$ everywhere except for in the stripes where instead $\kappa = 1$. We use homogeneous Neumann boundary conditions over the whole boundary, except for the two vertical sections on the left. On the top-most vertical part, $\Gamma_1$, we impose a nonhomogeneous Neumann condition $\kappa \nabla x \cdot n = \Psi u$ with $\Psi$ having the same hat-shaped form as in \nameref{example:BCC_grid}. On the bottom vertical part, $\Gamma_2$, we impose a homogeneous Dirichlet condition. These correspond to an insulated edge, a controllable heat input and a heat sink, respectively. The operator $\Bc$ is again given by $u \mapsto u\int_{\Gamma_1}{\Psi \varphi}$, and as output we take the mean of the temperature over the domain; $\Cc x = \int_{\Omega} x$. The meshes in this example have $n = 28, 84, 280, 1008, 3808, 14784$ interior nodes, respectively, while the reference solution uses $n=58240$.

\begin{figure}[t!]
\centering
\includegraphics[width=0.6\columnwidth]{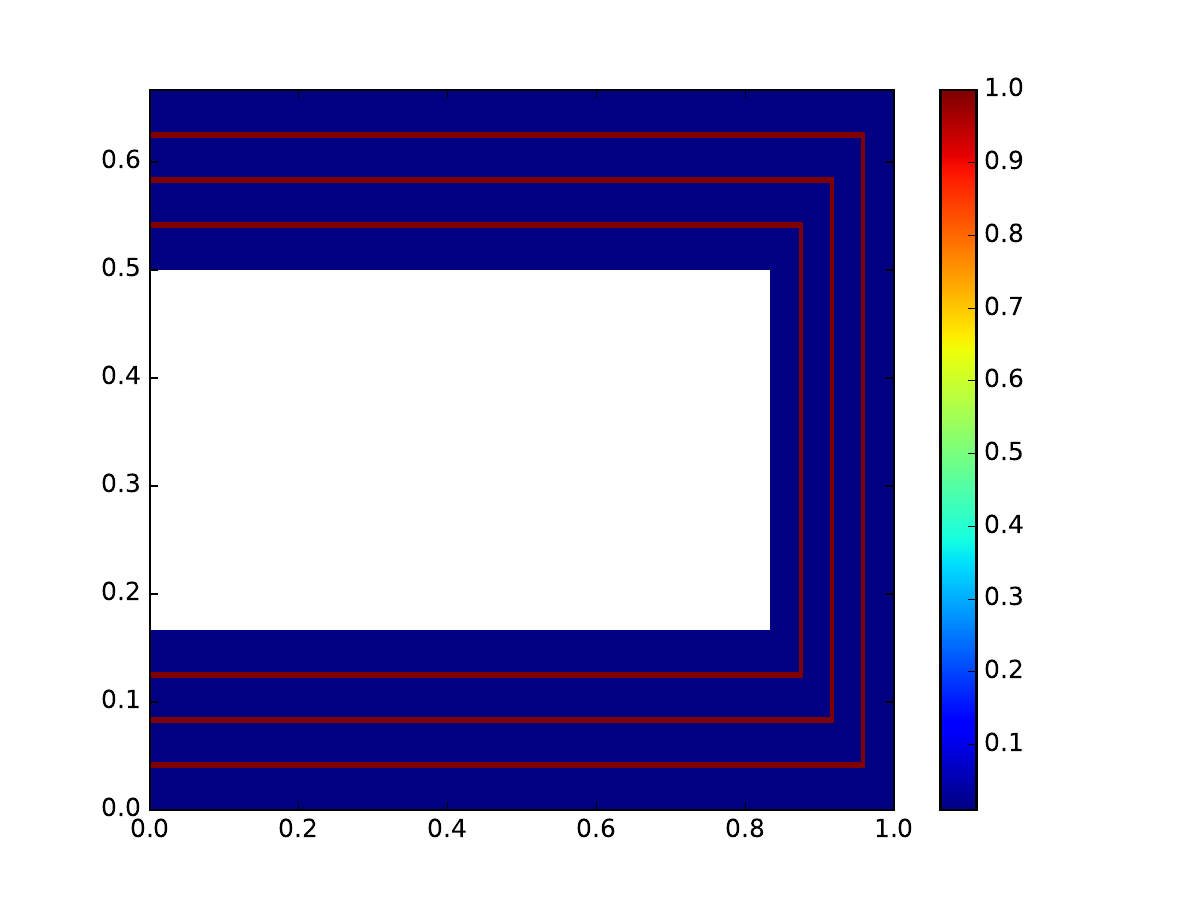}
\caption{The diffusion coefficient used in \nameref{example:BCC_Ushape}, plotted over the domain $\Omega$. (This figure is in color in the electronic version of the article.)}
\label{fig:diff_coeff_Ushape}
\end{figure}

The results are plotted in \cref{fig:BCC_Ushape}, where we can once again observe error behaviour consistent with the bounds given in \cref{thm:main_error_L2}.

\begin{figure}[t!]
\centering
\includegraphics[width=0.49\columnwidth]{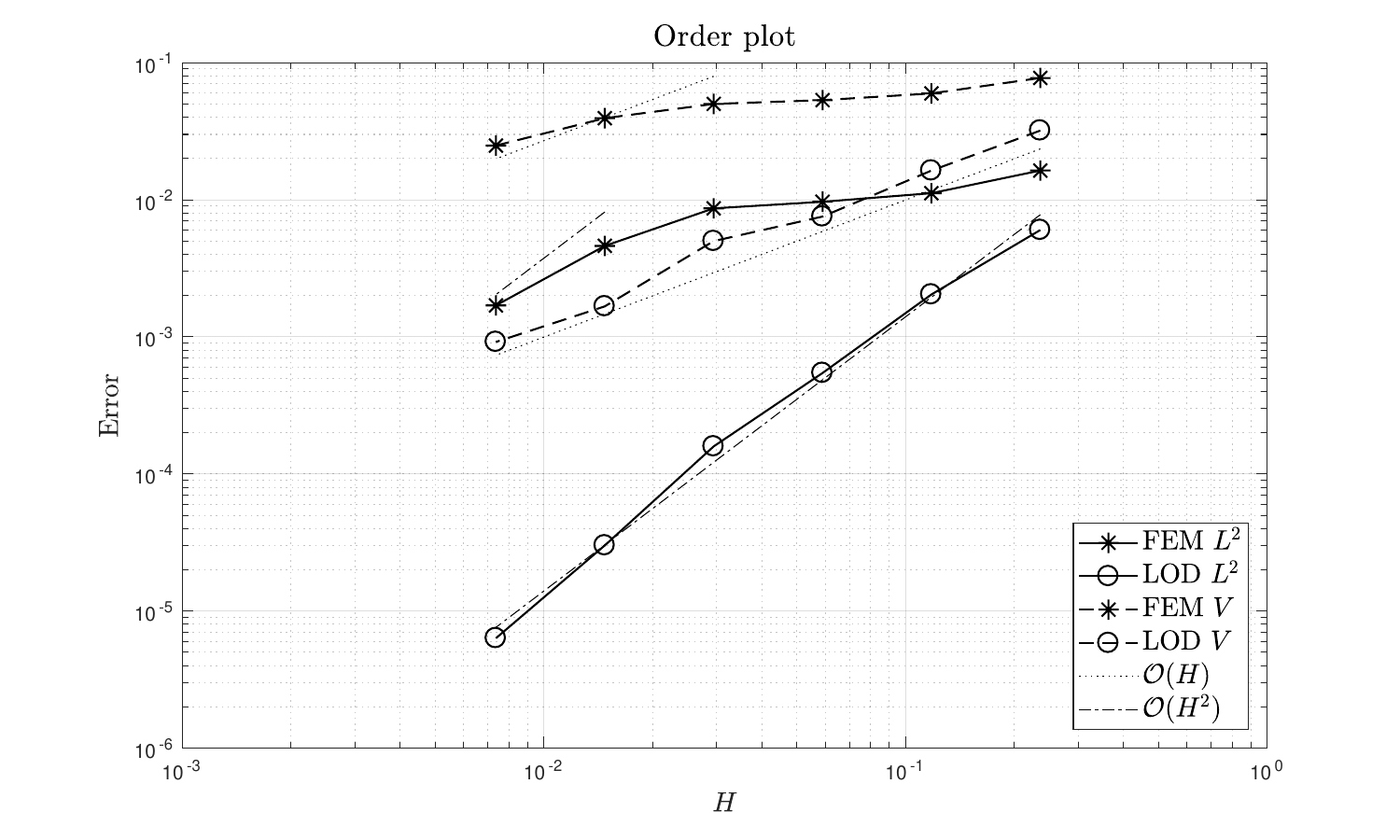}
\includegraphics[width=0.49\columnwidth]{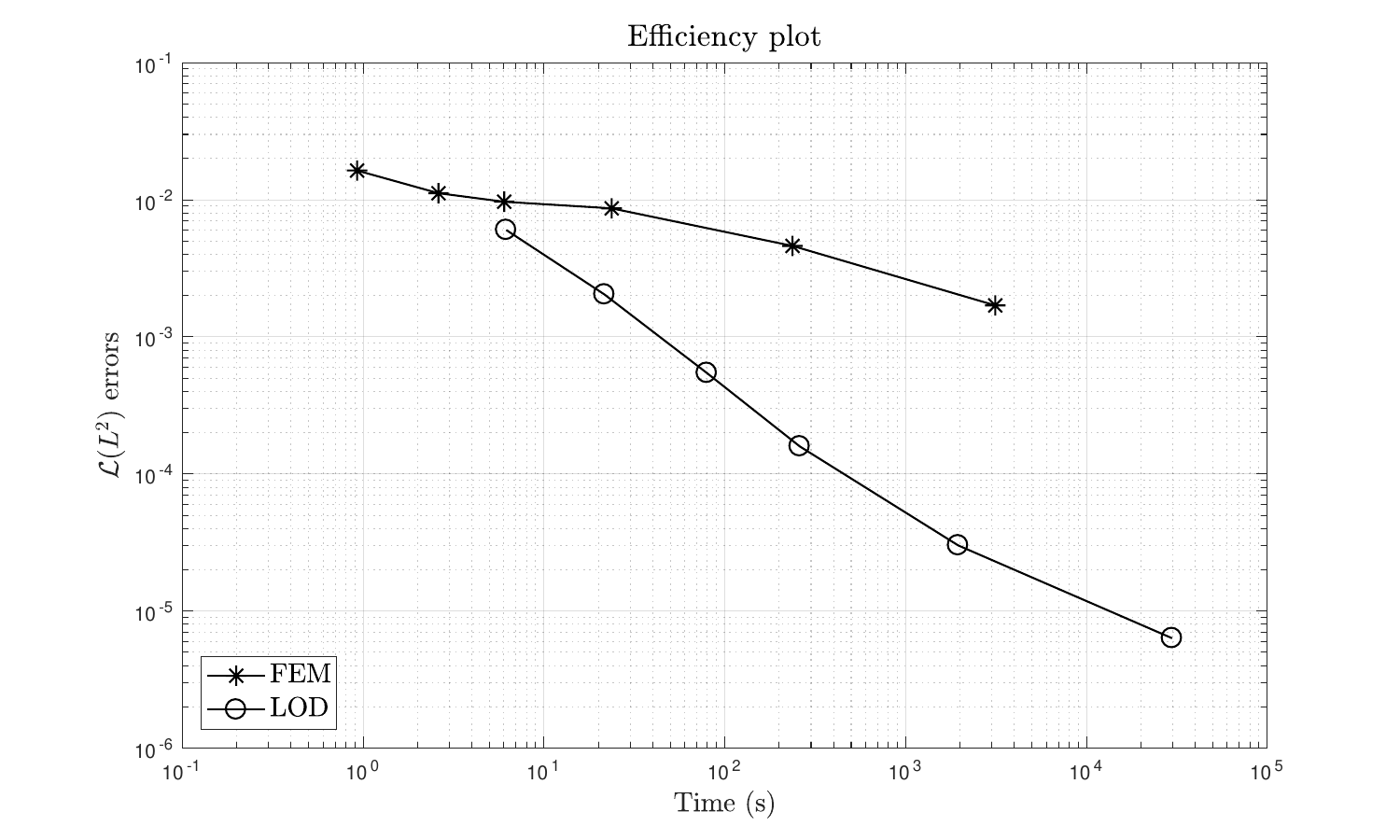}
\caption{Left: The $\LHH$- and $\LVV$-norm errors of the approximations computed in \nameref{example:BCC_Ushape}, plotted against the meshwidth. Right: The $\LHH$-norm errors plotted against the computation time.}
\label{fig:BCC_Ushape}
\end{figure}

\begin{remark}
  In all the experiments, we have chosen the fine-scale structure of the multiscale coefficient such that the reference FEM solution can resolve it, since otherwise we can not properly compute the respective errors. Decreasing the size of the fine-scale features even further would mean that the FEM convergence  is further delayed, while we may still compute accurate LOD approximations. In such a case, the efficiency of LOD in comparison to FEM is further (greatly) improved.
\end{remark}

\begin{remark}
We note that the finest discretizations of the demonstrated numerical experiments are representative of large-scale DRE problems. While there is of course no strict limit, the authors would at the time of writing classify large-scale as problems of size $n \ge 10^4$. Due to the matrix-valued nature of the equations, this is naively equivalent to solving a vector-valued differential equation of size $10^8$. While we do not employ a naive method, the number of unknowns still number in the millions. 
For readers more familiar with the theory of algebraic Riccati equations (AREs), i.e.\ the stationary counterparts of DREs, we note that one time step for a DRE solver is roughly equivalent to the solution of one corresponding ARE. The computational effort for solving a DRE is therefore usually at least two orders of magnitude higher, and large-scale in the ARE setting is therefore larger; starting rather  at around $n = 10^5$. We also note that the approximations were here computed on the equivalent of a modern desktop computer. With the increasing availability of parallelization on clusters or GPUs, we expect to see a shift towards even larger problems in the near future. However, for multiscale problems such as these, it is still critical to employ LOD.
\end{remark}

\section{Generalizations and future work} \label{sec:generalizations}
In this section we provide some notes on possible extensions of our theory and draw connections to related problems and methods.

\subsection{Boundary control} \label{subsec:BCC}
Boundary control applications such as \nameref{example:BCC_grid} occur frequently within the field of optimal control. Then either the input or output operator (or both) acts on the boundary of the computational domain.
In order to put such problems into the semigroup framework, one has to allow for unbounded operators $\Bc$ and $\Cc$~\cite{LasieckaTriggiani2000}. Clearly, our convergence analysis is no longer valid in that case, since we can no longer guarantee that $S_h \in \Lc(\Ltwo)$ or that $\Cc_h P_h \in \Lc(\Ltwo, Z)$. However, it is typically assumed that $\Bc$ and $\Cc$ are not \emph{too} unbounded. More specifically, if we suppose that $(-\Ac)^{-\beta}\Bc \in \Lc(U, \Ltwo)$ and $\Cc(-\Ac)^{-\gamma} \in \Lc(\Ltwo, Z)$, where $0 \le \beta + \gamma < 1$, we cover a large class of applications. 
Here,  $(-\Ac)^{-\alpha}$, denotes fractional powers of $\Ac$ which exist due to \cref{ass:operators}. They give rise to the spaces $X_{-\alpha} \supset \Ltwo$ as the completions of $\Ltwo$ in the norm $\norm{x}_{-\alpha} = \norm{(-\Ac)^{-\alpha}x}$. When $\gamma = 0$ we then have that $S_h \in \Lc(X_{-\beta})$, and by properly extending also the other involved operators to $X_{-\beta}$ we may follow the line of proof of \cref{thm:main_error_L2} and show convergence in $\Lc(X_{-\beta})$.

Obviously, this is a sub-optimal estimation, as $\norm{\cdot}_{-\beta}$ is a weaker norm than $\norm{\cdot}_{\Ltwo}$ for $\beta > 0$. However, from~\cite[Theorem 1.2.1.1]{LasieckaTriggiani2000} we have that $\Xt(t) S_h \Xt(t)$ is actually bounded in $\Ltwo$, at least away from $t = 0$. It therefore seems likely that one could use similar ideas to prove that the same holds for $\Yt(t) S_h \Yt(t)$, in which case we would have convergence in $\Lc(\Ltwo)$. Unfortunately, the theory required for such estimations is rather extensive, and we expect it to be even more so for the LOD approximations. We therefore leave such questions as future work.

\subsection{Systems of equations and applications in multiphysics}
In this paper we consider problems where the evolution operator $\Ac$ in the state equation defines an inner product of the form $a(u,v) = \int \kappa \nabla u \cdot \nabla v$. However, many interesting applications requires coupled systems to be modeled appropriately, for instance, multiphysical features such as thermoelasticity \cite{Biot56}, which describes temperature and displacement in a material. Another example is the singularly perturbed systems \cite{KaskiCoumarbatchGajic2010, MukaidaniXuMizukami2005}, which appear when modeling, for instance, fluid catalytic crackers. These are ill-conditioned problems due to a significantly larger time derivative for one (or more) of the equations.

The LOD method has successfully been applied to thermoelasticity and poroelasticity problems, see \cite{MalqvistPersson2017}. With more complicated models, the computational gain in using a coarse representation of the underlying partial differential equation is even greater. Analysis of such problems should be considered in the future.

\subsection{Other time discretizations} \label{subsec:other_timediscr}

It should also be noted that the LOD approach could be used with other time discretizations of the DRE. We have here chosen the Strang splitting scheme due to it being familiar to one of the authors and because an efficient implementation was readily available. However, there are also other types of splitting schemes~\cite{Stillfjord2017,Mena_etal2018}. Additionally, one might instead consider e.g.\ BDF and Rosenbrock methods~\cite{BennerMena2016,BennerMena2013,LangMenaSaak2015}, projection-based methods~\cite{KoskelaMena2017} or even peer methods~\cite{Lang2017}. These depend on solving linear equation systems rather than computing the solutions to parabolic problems, and the error analysis approach would thus differ. However, bounds similar to that given in \cref{lemma:parabolic_LOD} naturally exist also for stationary problems~\cite{MalqvistPeterseim2014}.

\subsection{Algebraic Riccati equations} \label{subsec:ARE}
The latter fact is even more relevant if one considers algebraic Riccati equations (AREs). These are the stationary counterpoints to the time-dependent DREs and arise when the final time $T$ in the cost functional goes to infinity. In this case, splitting does not apply, but we may still apply LOD to the equation to reduce its complexity. Then any method for AREs may be applied to solve this smaller problem, such as Newton-Kleinman ADI~\cite{BennerLiPenzl2008}, rational Krylov subspace methods~\cite{SimonciniSzyldMonsalve2014}, or RADI~\cite{BennerBujanovic_etal2018}. See also~\cite{BennerSaak2013} for a survey. Clearly, for each of these cases one would have to perform an error analysis such as the one provided in this paper.

 \section*{Acknowledgments}
 We are grateful to Fredrik Hellman for his assistance with the code for computing the LOD bases.


\begin{thebibliography}{10}

\bibitem{AbouKandil_etal2003}
{\sc H.~Abou-Kandil, G.~Freiling, V.~Ionescu, and G.~Jank}, {\em Matrix
  {R}iccati equations}, Systems \& Control: Foundations \& Applications,
  Birkh\"auser, Basel, 2003, \url{https://doi.org/10.1007/978-3-0348-8081-7}.

\bibitem{Bank2014}
{\sc R.~E. Bank and H.~Yserentant}, {\em On the {$H^1$}-stability of the
  {$L_2$}-projection onto finite element spaces}, Numer.\ Math., 126 (2014),
  pp.~361--381, \url{https://doi.org/10.1007/s00211-013-0562-4}.

\bibitem{BennerBujanovic_etal2018}
{\sc P.~Benner, Z.~Bujanovi\'c, P.~K\"urschner, and J.~Saak}, {\em R{ADI}: a
  low-rank {ADI}-type algorithm for large scale algebraic {R}iccati equations},
  Numer.\ Math., 138 (2018), pp.~301--330,
  \url{https://doi.org/10.1007/s00211-017-0907-5}.

\bibitem{BennerLiPenzl2008}
{\sc P.~Benner, J.-R. Li, and T.~Penzl}, {\em Numerical solution of large-scale
  {L}yapunov equations, {R}iccati equations, and linear-quadratic optimal
  control problems}, Numer.\ Linear Algebra Appl., 15 (2008), pp.~755--777,
  \url{https://doi.org/10.1002/nla.622}.

\bibitem{BennerMena2013}
{\sc P.~Benner and H.~Mena}, {\em Rosenbrock methods for solving {R}iccati
  differential equations}, IEEE Trans.\ Automat.\ Control, 58 (2013),
  pp.~2950--2956, \url{https://doi.org/10.1109/TAC.2013.2258495}.

\bibitem{BennerMena2016}
{\sc P.~Benner and H.~Mena}, {\em Numerical solution of the
  infinite-dimensional {LQR} problem and the associated {R}iccati differential
  equations}, J.~Numer.\ Math., 26 (2018), pp.~1--20,
  \url{https://doi.org/10.1515/jnma-2016-1039}.

\bibitem{BennerSaak2013}
{\sc P.~Benner and J.~Saak}, {\em Numerical solution of large and sparse
  continuous time algebraic matrix {R}iccati and {L}yapunov equations: a state
  of the art survey}, GAMM-Mitt., 36 (2013), pp.~32--52,
  \url{https://doi.org/10.1002/gamm.201310003}.

\bibitem{Bensoussan_etal_2007}
{\sc A.~Bensoussan, G.~Da~Prato, M.~C. Delfour, and S.~K. Mitter}, {\em
  Representation and control of infinite dimensional systems}, Systems \&
  Control: Foundations \& Applications, Birkh\"auser Boston, Inc., Boston, MA,
  second~ed., 2007, \url{https://doi.org/10.1007/978-0-8176-4581-6}.

\bibitem{Biot56}
{\sc M.~A. Biot}, {\em Thermoelasticity and irreversible thermodynamics}, J.\
  Appl.\ Phys., 27 (1956), pp.~240--253.

\bibitem{BosargeJohnsonSmith1973}
{\sc W.~E. Bosarge, Jr., O.~G. Johnson, and C.~L. Smith}, {\em A direct method
  approximation to the linear parabolic regulator problem over multivariate
  spline bases}, SIAM J.~Numer.\ Anal., 10 (1973), pp.~35--49,
  \url{https://doi.org/10.1137/0710006}.

\bibitem{CaoLiuAllegrettoLin2012}
{\sc L.~Cao, J.~Liu, W.~Allegretto, and Y.~Lin}, {\em A multiscale approach for
  optimal control problems of linear parabolic equations}, SIAM J.~Control
  Optim., 50 (2012), pp.~3269--3291, \url{https://doi.org/10.1137/110828800}.

\bibitem{ChenHuangLiuYan2015}
{\sc Y.~Chen, Y.~Huang, W.~Liu, and N.~Yan}, {\em A mixed multiscale finite
  element method for convex optimal control problems with oscillating
  coefficients}, Comput.\ Math.\ Appl., 70 (2015), pp.~297--313,
  \url{https://doi.org/10.1016/j.camwa.2015.03.020}.

\bibitem{Engwer_etal_2016}
{\sc C.~{Engwer}, P.~{Henning}, A.~{M{\aa}lqvist}, and D.~{Peterseim}}, {\em
  {Efficient implementation of the Localized Orthogonal Decomposition method}},
  ArXiv e-prints,  (2016), \url{https://arxiv.org/abs/1602.01658}.
\newblock https://arxiv.org/abs/1602.01658.

\bibitem{Falk1973}
{\sc R.~S. Falk}, {\em Approximation of a class of optimal control problems
  with order of convergence estimates}, J.~Math.\ Anal.\ Appl., 44 (1973),
  pp.~28--47, \url{https://doi.org/10.1016/0022-247X(73)90022-X}.

\bibitem{HenningMalqvist2014}
{\sc P.~Henning and A.~M{\aa}lqvist}, {\em Localized orthogonal decomposition
  techniques for boundary value problems}, SIAM J. Sci. Comput., 36 (2014),
  pp.~A1609--A1634, \url{https://doi.org/10.1137/130933198}.

\bibitem{KoskelaMena2017}
{\sc A.~Koskela and H.~Mena}, {\em {Analysis of {K}rylov Subspace Approximation
  to Large Scale Differential {R}iccati Equations}}, ArXiv e-prints,  (2017),
  \url{https://arxiv.org/abs/1705.07507},
  \url{https://arxiv.org/abs/1705.07507}.

\bibitem{KaskiCoumarbatchGajic2010}
{\sc S.~Koskie, C.~Coumarbatch, and Z.~Gajic}, {\em Exact slow-fast
  decomposition of the singularly perturbed matrix differential {R}iccati
  equation}, Appl.\ Math.\ Comput., 216 (2010), pp.~1401--1411,
  \url{https://doi.org/10.1016/j.amc.2010.02.040}.

\bibitem{KrollerKunisch1991}
{\sc M.~Kroller and K.~Kunisch}, {\em Convergence rates for the feedback
  operators arising in the linear quadratic regulator problem governed by
  parabolic equations}, SIAM J.~Numer.\ Anal., 28 (1991), pp.~1350--1385,
  \url{https://doi.org/10.1137/0728071}.

\bibitem{Lang2017}
{\sc N.~Lang}, {\em Numerical Methods for Large-Scale Linear Time-Varying
  Control Systems and related Differential Matrix Equations}, dissertation,
  Technische Universit{\"a}t Chemnitz, Chemnitz, Germany, June 2017.

\bibitem{LangMenaSaak2015}
{\sc N.~Lang, H.~Mena, and J.~Saak}, {\em On the benefits of the {$LDL^T$}
  factorization for large-scale differential matrix equation solvers}, Linear
  Algebra Appl., 480 (2015), pp.~44--71,
  \url{https://doi.org/10.1016/j.laa.2015.04.006}.

\bibitem{LasieckaTriggiani2000}
{\sc I.~Lasiecka and R.~Triggiani}, {\em Control theory for partial
  differential equations: continuous and approximation theories. {I}}, vol.~74
  of Encyclopedia of Mathematics and its Applications, Cambridge University
  Press, Cambridge, 2000.
\newblock Abstract parabolic systems.

\bibitem{Li2010}
{\sc J.~Li}, {\em A multiscale finite element method for optimal control
  problems governed by the elliptic homogenization equations}, Comput.\ Math.\
  Appl., 60 (2010), pp.~390--398,
  \url{https://doi.org/10.1016/j.camwa.2010.04.017}.

\bibitem{MalqvistPersson2017}
{\sc A.~M{\aa}lqvist and A.~Persson}, {\em A generalized finite element method
  for linear thermoelasticity}, ESAIM Math.\ Model.\ Numer.\ Anal., 51 (2017),
  pp.~1145--1171, \url{https://doi.org/10.1051/m2an/2016054}.

\bibitem{MalqvistPersson_2015}
{\sc A.~M{\aa}lqvist and A.~Persson}, {\em Multiscale techniques for parabolic
  equations}, Numer. Math., 138 (2018), pp.~191--217,
  \url{https://doi.org/10.1007/s00211-017-0905-7}.

\bibitem{MalqvistPeterseim_2014}
{\sc A.~M{\aa}lqvist and D.~Peterseim}, {\em Localization of elliptic
  multiscale problems}, Math.\ Comp., 83 (2014), pp.~2583--2603.

\bibitem{MalqvistPeterseim2014}
{\sc A.~M{\aa}lqvist and D.~Peterseim}, {\em Localization of elliptic
  multiscale problems}, Math.\ Comp., 83 (2014), pp.~2583--2603,
  \url{https://doi.org/10.1090/S0025-5718-2014-02868-8}.

\bibitem{McKnightBosarge1973}
{\sc R.~S. McKnight and W.~E. Bosarge, Jr.}, {\em The {R}itz-{G}alerkin
  procedure for parabolic control problems}, SIAM J.~Control Optim., 11 (1973),
  pp.~510--524.

\bibitem{Mena_etal2018}
{\sc H.~Mena, A.~Ostermann, L.-M. Pfurtscheller, and C.~Piazzola}, {\em
  Numerical low-rank approximation of matrix differential equations},
  J.~Comput.\ Appl.\ Math.,  (2018), \url{https://arxiv.org/abs/1705.10175}.
\newblock Accepted for publication.

\bibitem{MukaidaniXuMizukami2005}
{\sc H.~Mukaidani, H.~Xu, and K.~Mizukami}, {\em Numerical algorithm for
  solving cross-coupled algebraic {R}iccati equations of singularly perturbed
  systems}, in Advances in dynamic games, vol.~7 of Ann. Internat. Soc. Dynam.
  Games, Birkh\"auser Boston, Boston, MA, 2005, pp.~545--570,
  \url{https://doi.org/10.1007/0-8176-4429-6_29}.

\bibitem{Rosen1991}
{\sc I.~G. Rosen}, {\em Convergence of {G}alerkin approximations for operator
  {R}iccati equations---a nonlinear evolution equation approach}, J.~Math.\
  Anal.\ Appl., 155 (1991), pp.~226--248,
  \url{https://doi.org/10.1016/0022-247X(91)90035-X}.

\bibitem{SimonciniSzyldMonsalve2014}
{\sc V.~Simoncini, D.~B. Szyld, and M.~Monsalve}, {\em On two numerical methods
  for the solution of large-scale algebraic {R}iccati equations}, IMA
  J.~Numer.\ Anal., 34 (2014), pp.~904--920,
  \url{https://doi.org/10.1093/imanum/drt015}.

\bibitem{Stillfjord2015}
{\sc T.~Stillfjord}, {\em Low-rank second-order splitting of large-scale
  differential {R}iccati equations}, IEEE Trans.\ Automat.\ Control, 60 (2015),
  pp.~2791--2796, \url{https://doi.org/10.1109/TAC.2015.2398889}.

\bibitem{Stillfjord2017}
{\sc T.~Stillfjord}, {\em Adaptive high-order splitting schemes for large-scale
  differential {R}iccati equations}, Numer.\ Algorithms,  (2017),
  \url{https://doi.org/10.1007/s11075-017-0416-8}.

\bibitem{Tanabe1979}
{\sc H.~Tanabe}, {\em Equations of evolution}, vol.~6 of Monographs and Studies
  in Mathematics, Pitman (Advanced Publishing Program), Boston, Mass.-London,
  1979.
\newblock Translated from the Japanese by N. Mugibayashi and H. Haneda.

\bibitem{Thomee2006}
{\sc V.~Thom\'{e}e}, {\em {Galerkin Finite Element Methods for Parabolic
  Problems}}, Springer Series in Computational Mathematics, Springer-Verlag,
  Berlin, second~ed., 2006.

\bibitem{Winther1978}
{\sc R.~Winther}, {\em Error estimates for a {G}alerkin approximation of a
  parabolic control problem}, Ann.\ Mat.\ Pura Appl.\ (4), 117 (1978),
  pp.~173--206, \url{https://doi.org/10.1007/BF02417890}.

\end{thebibliography}
\end{document}